\documentclass{article}  
\usepackage[mathscr]{euscript}

\newcommand{\la}{\lambda}
\newcommand{\La}{\Lambda}
\newcommand{\al}{\alpha}
\newcommand{\til}{\tilde}
\newcommand{\ti}{\times}
\newcommand{\ga}{\gamma}
\newcommand{\Ga}{\Gamma}
\newcommand{\Om}{\Omega}
\newcommand{\sig}{\sigma}

\newcommand{\eps}{\epsilon}
\newcommand{\var}{\varphi}
\newcommand{\del}{\delta}
\newcommand{\Del}{\Delta}
\newcommand{\f}{\frac}

\newcommand{\subs}{\subseteq}
\newcommand{\ovs}{\overset}

\newcommand{\BN}{\Bbb{N}}
\newcommand{\BZ}{\Bbb{Z}}
\newcommand{\BR}{\Bbb{R}}
\newcommand{\BC}{\Bbb{C}}

\newcommand{\cH}{{\cal{H}}}

\newcommand{\lo}{\longrightarrow}
\newcommand{\Lo}{\Longrightarrow}
\newcommand{\LO}{\Longleftrightarrow}
\newcommand{\rig}{\rightarrow}

\newcommand{\Max}{\operatorname{Max}}

\newcommand{\Sign}{\operatorname{Sign}}
\newcommand{\bP}{{\bold{P}}}

\newcommand{\bS}{{\bold{S}}}
\newcommand{\bT}{{\bold{T}}}
\newcommand{\bA}{{\bold{A}}}
\newcommand{\sA}{\mathscr{A}}

\input{epsf}

    \usepackage{graphicx}
    \usepackage{amsmath}
    \usepackage{amsthm}
    \usepackage{amssymb}
    \usepackage{latexsym}

\newtheorem{theorem}{Theorem}[section]
\newtheorem{lemma}[theorem]{Lemma}
\newtheorem{proposition}[theorem]{Proposition}
\newtheorem{corollary}[theorem]{Corollary}
\theoremstyle{definition}
\newtheorem{definition}[theorem]{Definition}

\newtheorem{remark}[theorem]{Remark}

\title{An iterative thresholding algorithm for linear inverse problems with 
mixed multi-constraints and its applications}
\author{Saman Khoramian\thanks{{\it E-mail address:}
 saman.khoramian@gmail.com}}
\date{ }

\begin{document}

\maketitle

\begin{abstract}
In this paper, we will present a generalization for a minimization problem 
from I. Daubechies, M. Defrise, and C. Demol [3].
This generalization is useful for solving many practical problems in which 
more than one constraint are involved. In this regard, we will conclude the 
findings of many papers (most of which are on image processing) from this 
generalization. It is hoped that the approach proposed in this paper will be a 
suitable reference for some applied works where multi-frames, multi-wavelets, 
or multi-constraints are present in linear inverse problems.
\end{abstract}

\section{Introduction}
Many practical problems in sciences especially in 
 applied sciences, can expressed as the 
following operator equation 
$$Kf=h$$
where $K:X\lo Y$ is a linear bounded operator between normed spaces $X,Y$, and 
$h\in Y$ is fixed.

The observations or {\it data}
 are typically not exactly equal to $h=Kf$, but 
rather to a distortion  of $h$. This distortion is often modeled by an 
{\it additive noise} or {\it error} term $e$, i.e.
$$g=h+e=Kf+e.$$
To find an estimate of $f$ from the observed $g$, one can minimize the 
{\it discrepancy} $\Del(f)$,
$$\Del(f)=\|Kf-g\|^2$$
In some problems, we have a priori knowledge about the solution. For instance, 
we know $\|f\|=\rho$ or $\|f\|\leq \rho$, where $\rho$ is constant. In such 
situations, we define the following functional 
$$\Phi(f;g)=\|Kf-g\|^2+\mu\|f\|^2$$
where $\mu$ is some positive constant called the {\it
 regularization parameter.}
Minimization of the following functionals has been considered in [1,2,6].
\begin{align*}
& \|Kf-g\|^2_Y+ \mu\|f\|_X^2,\\
& \|I(f)-g\|^2_{L_2(I)}+\la \|f\|^2_{W^\beta(L_2(I))}
; \|f\|^2_{W^\beta((L_2(I))}=\sum_{0\leq k} 
\sum_{j\in\BZ_k^2}\sum_{\psi\in \psi_k} 2^{2\beta k}|c_{j,k,\psi}|^2; 
\beta<\f{1}{2}\\
& \|I(f)-g\|^2_{L_2(I)}+\la \|f\|^\tau_{B^\beta_\tau(L_\tau(I))}
; \|f\|^\tau_{B^\beta_\tau((L_\tau(I))}=\sum_{0\leq k} 
\sum_{j\in\BZ_k^2}\sum_{\psi\in \psi_k} |c_{j,k,\psi}|^\tau 
\end{align*}
All of above functionals are special cases of the following functional 
$$\Phi(f)=\|Kf-g\|^2_{\cH'}+|\!|\!|f|\!|\!|^p_{W,p}$$
where $\cH,\cH'$ are Hilbert spaces and $K$ is linear bounded operator 
from $\cH$ to $\cH'$ and 
$$|\!|\!|f|\!|\!|_{W,p}=\left( \sum_{\ga\in\Ga} w_\ga | 
<f,\var_\ga>|^p\right)^{\f{1}{p}}$$
for $1\leq p\leq 2$, is a weighted $l^p$-norm of the coefficients of $f$ with 
respect to an orthonormal basis $(\var_\ga)_{\ga\in\Ga}$ of $\cH$, and a 
sequence of strictly positive weights $W=(w_\ga)_{\ga\in\Ga}$. Studying the
minimizer of $\Phi$ is useful for some applied sciences (see [3]). The 
minimizer of $\Phi$ has not been found, but $I$. Daubechies, M. Defrise, 
and C. 
Demol [3] have found a sequence in $\cH$ that converges strongly to the 
minimizer of $\Phi$. In this paper, we will do the same to the following 
functional, 
\begin{equation*}
\Phi(f)=\|Kf-g\|^2_{\cH'} 
+|\!|\!|f|\!|\!|^{p_1}_{W_1,p_1}+\dots+|\!|\!|f|\!|\!|^{p_n}_{W_n,p_n} \tag{$*$}
\end{equation*}
where 
$$|\!|\!|f|\!|\!|_{W_i,p_i}=\left( \sum_{\ga\in\Ga_i} w_\ga| <f,\var_\ga> 
|^{p_i}\right)^{\f{1}{p_i}}$$
for $1\leq i\leq n; 1\leq p_i\leq 2$, $W_i=\{w_\ga\}_{\ga\in\Ga_i}, 
\Ga=\Ga_1\cup\dots\cup\Ga_{n}$, $\{\var_\ga\}_{\ga\in\Ga}$ is an orthonormal 
basis for Hilbert space $\cH$.

We call $\sum_{i=1}^n |\!|\!|f|\!|\!|^{p_i}_{W_i,p_i}$ {\it mixed  
multi-constraints} 
and show them by the following shorthand notation, 
$$|\!|\!|f|\!|\!|_{W,\bP}^{\bP} = 
|\!|\!|f|\!|\!|^{p_1}_{W_1,p_1}+\dots+|\!|\!|f|\!|\!|^{p_n}_{W_n,p_n}$$
such that $\bP=\{p_1,p_2,\dots,p_n\}$ and $W=W_1\cup\dots\cup W_n$.

In the special case of $p=p_1=p_2=\dots=p_n$, we will have
$$|\!|\!|f|\!|\!|_{W,\bP}^{\bP}=\sum_{i=1}^n 
|\!|\!|f|\!|\!|^{p_i}_{W_i,p_i}=\sum_{i=1}^n 
\sum_{\ga\in\Ga_i} w_\ga| <f,\var_\ga>|^p=\sum_{\ga\in\Ga} w_\ga 
|<f,\var_\ga>|^p=|\!|\!|f|\!|\!|^p_{W,p}$$
For regularization, we will define for $\al=(\al_1,\dots,\al_n)$ the 
functional $\Phi_{\al;g}$ on $\cH$ by 
$$\Phi_{\al;g}(f)=\|Kf-g\|^2 
+\al_1|\!|\!|f|\!|\!|^{p_1}_{W_1,p_1}+\al_2|\!|\!|f|\!|\!|^{p_2}_{W_2,p_2}+
\dots+\al_n |\!|\!|f|\!|\!|^{p_n}_{W_n,p_n}$$
Also we
 will suppose $f^*_{\al;g}$ is minimizer of the functional $\Phi_{\al;g}$. By 
assuming $$\al(\eps)=(\al_1(\eps),\al_2(\eps),\dots,\al_n(\eps))$$
 such that 
$$\lim_{\eps\rig 0} \al_i(\eps)=0,\quad \lim_{\eps\rig 0} 
\f{\eps^2}{\al_i(\eps)}=0,\quad \lim_{\eps\rig 0} 
\f{\al_i(\eps)}{\al_j(\eps)}=1$$
for every $i,j; 1\leq i,j\leq n$, we will show, for any $f_0\in \cH$,
$$\lim_{\eps\rig 0} \left[ \sup_{\|g-Kf_0\|<\eps} 
\|f^*_{\al(\eps);g}-f^\dagger\|\right]=0$$
where $f^\dagger$ is the unique minimal element  with regard to 
$|\!|\!|\cdot|\!|\!|^{\bP}_{W,\bP}$ in $S=N(K)+f_0=\{f;K(f)=K(f_0)\}$.

Here, we adopt the same procedure as in I. Daubechies, M. Defrise, and C.DeMol 
[3]. Although the novelty of the paper lies in the section related to 
regularization, since changing $|\!|\!|\cdot|\!|\!|^{p}_{W,p}$ to 
$|\!|\!|\cdot|\!|\!|^{p_1}_{W_1,p_1}+\dots+
|\!|\!|\cdot|\!|\!|^{p_n}_{W_n,p_n}$ is a fundamental change, in other 
sections we repeat the procedure adopted in [3] for accuracy, and provide the 
proofs of the theorems to which some change, albeit minor, has been made. 
Moreover, for a more logical presentation of the material, the order of the 
theorems has been slightly changed.

This generalization can be used in applied problems. To show this in section 
3, we conclude the finding of the following papers from this generalization.

\begin{itemize}
\item[{[4]}] 
I.Daubechies, G. Teschke, Variational image restoration by means of wavelets: 
Simultaneous decomposition, deblurring and denoising, Appl. Comput. Harmon. 
Anal. 19(1) (2005) 1-16.
\item[{[5]}]
 M. Defrise, C. Demol, Inverse imaging with mixed penalties, in: Conference 
Proceedings, 2004.
\item[{[9]}]
 G. Teschke, Multi-frame representations in linear inverse problems 
with mixed multi-constraints, Appl. Comput. Harmon. Anal. 22(2007), 43-60. 
\end{itemize}

Furthermore, we point out that in subsection 3.1 we have managed to prove a 
regularization theorem required by G. Teschke [7].

\section{An iterative thresholding algorithm for linear inverse problems with 
mixed multi-constraints}
\subsection{An iterative algorithm through surrogate functionals}
At first, we present a Lemma which has been proved in [3]:
\begin{lemma}\label{2.1}
The minimizer of the function $M(x)=x^2-2bx+c|x|^p$ for $p\geq 1$, $c>0$ is 
$S_{c,p}(b)$, where
 the function $S_{c,p}$ from $R$ to itself is defined by 
$$S_{c,p}(t)=\begin{cases}
F^{-1}_{c,p}(t) & p>1\\
t-\f{c}{2} & p=1,t> \f{c}{2}\\
0 & p=1, |t|\leq \f{c}{2} \\
t+\f{c}{2} & p=1,t< -\f{c}{2}
\end{cases}$$
where the function $F_{c,p}$ is defined by 
$$F_{c,p}(t)=t+\f{cp}{2} \Sign (t) |t|^{p-1}.$$
\end{lemma}
In this paper, we use the shorthand notation $f_\ga$ for $<f,\var_\ga>$, 
$h_\ga$ for $<h,\var_\ga>$, etc.

\begin{proposition}
Suppose $K:\cH\lo \cH'$ is an operator, with $\|KK^*\|<1$,
$(\var_\ga)_{\ga\in\Ga}$ is an orthonormal basis for $\cH$, and 
$W=(w_\ga)_{\ga\in\Ga}$ is a sequence such that $\forall \ga\in\Ga$ 
$w_\ga>c>0$. Further suppose $g$ is an element of $\cH'$. Let 
$\Ga=\Ga_1\cup\Ga_2\cup \dots\cup \Ga_n$, $W=W_1\cup W_2\cup\dots\cup W_n$, 
$\bP =\{p_1,\dots,p_n\}$ such that $p_i\geq 1$ for $1\leq i\leq n$. 
Choose $a\in 
\cH$ and define the functional $\Phi^{S\cup R}_{W,\bP}(f;a)$ on $\cH$ by 
$$\Phi^{S\cup R}_{W,\bP} (f;a)=\|Kf-g\|^2+ \sum_{i=1}^n \sum_{\ga\in\Ga_i} 
w_\ga |f_\ga|^{p_i}+ \|f-a\|^2 - \|K(f-a)\|^2.$$
Also, define operators $\bS_{W,\bP}$ by 
$$\bS_{W,\bP}(h)=\sum_{i=1}^n \sum_{\ga\in\Ga_i}S_{w_\ga,p_i}(h_\ga)\var_\ga$$
with functions $S_{w,p}$ from $R$ to itself given by Lemma 2.1.

By these assumptions, we will have 

A) $f_{\min}$=minimizer of the functional $\Phi_{W,\bP}^{S\cup 
R}=\bS_{W,\bP}(a+K^*(g-Ka))$

B) for all $h\in \cH$, one has
$$\Phi^{S\cup R}_{W,\bP}(f_{\min}+h;a) \geq \Phi^{S\cup 
R}_{W,\bP}(f_{\min};a)+\|h\|^2.$$
\end{proposition}

\begin{proof}
A) 
\begin{align*}
\Phi^{S\cup R}_{W,\bP}(f;a) &=\|Kf-g\|^2 +\sum_{\ga\in\Ga} 
w_\ga|<f,\var_\ga>|^{p_\ga} -\|Kf-Ka\|^2+\|f-a\|^2\\
&=\|f\|^2-2<f,a+K^*g-K^*Ka>+\sum_{\ga\in\Ga} w_\ga |f_\ga|^{p_\ga} +\|g\|^2+ 
\|a\|^2- \|Ka\|^2\\
&=\sum_{\ga\in\Ga} [|f_\ga|^2 -2f_\ga(a+K^*g-K^*Ka)_\ga+ w_\ga|f_\ga|^{p_\ga}] 
+ \|g\|^2+\|a\|^2-\|Ka\|^2\\
&=\sum_{i=1}^n \sum_{\ga\in\Ga_i} 
[f^2_\ga-2f_\ga(a+K^*g-K^*Ka)_\ga+w_\ga|f_\ga|^{p_i}] +\|g\|^2+\|a\|^2-\|Ka\|^2
\end{align*}
By Lemma 2.1, we have
\begin{align*}
f_{\min} &=\sum_{i=1}^n \left( \sum_{\ga\in\Ga_i} 
S_{w_\ga,p_i}((a+K^*g-K^*Ka)_\ga)\var_\ga\right) \\
&=\sum_{i=1}^n \bS_{W_i,p_i} (a+K^*g-K^*Ka)=\bS_{W,\bP}(a+K^*g-K^*Ka)
\end{align*}

B) 
\begin{align*}
\Phi_{W,\bP}^{S\cup R}(f+h;a) &-\Phi_{W,\bP}^{S\cup R} (f;a) \\
&=\sum_{\ga\in\Ga} w_\ga 
(|f_\ga+h_\ga|^{p_\ga}-|f_\ga|^{p_\ga})+2<h,f-a-K^*(g-Ka)>+\|h\|^2\\
&=\sum_{i=1}^n \left( \sum_{\ga\in\Ga_i} 
w_\ga|f_\ga+h_\ga|^{p_i}-w_\ga|f_\ga|^{p_i}\right) + \sum_{\ga\in\Ga} 
2h_\ga(f-a-K^*(g-Ka))_\ga+\|h\|^2\\
&=\sum_{i=1}^n \left( \sum_{\ga\in\Ga_i} 
w_\ga|f_\ga+h_\ga|^{p_i}-w_\ga|f_\ga|^{p_i} 
+2h_\ga(f-a-K^*(g-Ka))_\ga\right)+\|h\|^2
\end{align*}
In [3], the following has been established for $p_i=1$ 
$$w_\ga|f_\ga+h_\ga|^{p_i} -w_\ga|f_\ga|^{p_i} +2h_\ga(f-a-K^*(g-Ka))_\ga\geq 
0$$
If $p_i>1$, since $f_\ga=S_{w_\ga,p_i}((a+K^*g-K^*Ka)_\ga)$, we have
$$2(f-a-K^*(g-Ka))_\ga=-w_\ga p_i\Sign (f_\ga)|f_\ga|^{p_i-1}$$
Also, for $f_\ga\neq 0$, there is  $\al\in R$ such that $h_\ga=\al f_\ga$. Then
\begin{align*}
w_\ga |f_\ga+h_\ga|^{p_i} - & w_\ga|f_\ga|^{p_i} +2h_\ga(f-a-K^*(g-Ka))_\ga\\
&=w_\ga|f_\ga+h_\ga|^{p_i}-w_\ga |f_\ga|^{p_i} -w_\ga p_i h_\ga \Sign 
(f_\ga)|f_\ga|^{p_i-1}\\
&=w_\ga|f_\ga+\al f_\ga|^{p_i}-w_\ga|f_\ga|^{p_i} -p_i \al w_\ga f_\ga 
\Sign(f_\ga)|f_\ga|^{p_i-1}\\
&=w_\ga|f_\ga|^{p_i}(|1+\al|^{p_i}-1-p_i\al)\geq 0.
\end{align*}
\end{proof}

\begin{definition}\label{2.3}
Pick $f^0$ in $\cH$. We define the functions $f^n$
 recursively by the following algorithm:
\begin{align*}
f^1 &=\arg-\min(\Phi^{S\cup R}_{W,\bP}(f;f^0)), \quad f^2=\arg-\min(\Phi^{S\cup
R}_{W,\bP}(f,f^1)),\dots,\\
f^n &=\arg-\min(\Phi^{S\cup R}_{W,\bP}(f;f^{n-1})).
\end{align*}
\end{definition}

\begin{corollary}\label{2.4}
By Definition 2.3, we clearly have
$$f^n=\bS_{W,\bP}(f^{n-1}+K^*(g-Kf^{n-1})).$$
\end{corollary}

\begin{definition}\label{2.5}
We define operator $\bT$ from $\cH$ to $\cH$ by 
$$\bT(f)=\bS_{W,\bP}(f+K^*(g-Kf)).$$
\end{definition}

\begin{corollary}\label{2.6}
By Definition 2.5, we have
$$\bT^n(f^0)=f^n.$$
\end{corollary}

\subsection{Weak convergence of the $f^n$}
The following lemma can be proved in the same manner as its corresponding
lemma in [3]:
\begin{lemma}\label{2.7}
If $f^*\in \{f\in\cH|\bT(f)=f\}$ then $f^*$ is a minimizer of $\Phi_{W,\bP}$.
\end{lemma}

By the following lemma, we state characteristics of function $T$ and the 
sequence $\{f^n\}_{n=1}^\infty$ which were introduced in Definition 2.3 and 
Definition 2.5.

\begin{lemma}\label{2.8}
Let $K$ be a bounded linear operator from $\cH$ to $\cH'$, with 
the norm strictly bounded by 1. 
Take $p_i\in [1,2]$ for $1\leq i\leq n$, and let $\bS_{W,\bP}$ be the operator 
defined by 
$$\bS_{W,\bP}(h)=\sum_{\ga\in\Ga} S_{w_\ga,p_\ga}(h_\ga)\var_\ga$$
where the sequence $W=(w_\ga)_{\ga\in \Ga}$ is uniformly bounded below away 
from zero, i.e. there exists a constant $c>0$ such that $\forall \ga\in\Ga; 
w_\ga\geq c$ and we have $p_i=p_\ga$ for every $\ga\in\Ga_i$. Let,
$\Ga=\Ga_1\cup \Ga_2\cup \dots\cup \Ga_{n}, W_i =\{w_\ga\}_{\ga\in\Ga_i},$ 
$\bP=\{p_1,p_2,\dots,p_n\}$ and suppose 
$|\!|\!|f|\!|\!|^{p_i}_{W_i,p_i}=\sum_{\ga\in\Ga_i} w_\ga|f_\ga|^{p_i}$,
\begin{align*}
|\!|\!|f|\!|\!|^{\bP}_{W,\bP} &=\sum_{\ga\in\Ga} w_\ga |f_\ga|^{p_\ga} \\
&=\sum_{i=1}^n \sum_{\ga\in\Ga_i}w_\ga |f_\ga|^{p_i} =|\!|\!| 
f|\!|\!|^{p_1}_{W_1,p_1}+\dots+ |\!|\!|f|\!|\!|^{p_n}_{W_n,p_n} 
\end{align*}
\begin{align*}
\Phi_{W,\bP}(f) &=\|Kf-g\|^2+|\!|\!|f|\!|\!|^{\bP}_{W,\bP},\\
\Phi^{S\cup R}_{W,\bP}(f)&=\Phi_{W,\bP}(f)+\|f-a\|^2-\|K(f-a)\|^2
\end{align*}
Then 

A) $\forall v,v'\in \cH$; $\|\bS_{W,\bP}(v)-\bS_{W,\bP}(v')\|\leq \|v-v'\|$

B) $\forall v,v'\in \cH; \|\bT(v)-\bT(v')\|\leq \|v-v'\|$

C) Both $(\Phi_{W,\bP}(f^n))_{n\in\BN}$ and $(\Phi_{W,\bP}^{S\cup 
R}(f^{n+1};f^n))_{n\in\BN}$ are non-increasing sequences.

D) $\exists M>0; \forall n\in N; \|f^n\|<M$

E) $\sum_{n=0}^\infty \|f^{n+1}-f^n\|^2<\infty$

F) $\|\bT^{n+1}(f^0)-\bT^n(f^0)\|=\|f^{n+1}-f^n\|\lo 0$ for $n\lo\infty$.

Parts B,C,E,F have similar proofs to their corresponding lemmas in [3]; thus,
we just prove parts A,D, which have some  differences in their proofs.
\end{lemma}

\begin{proof} (A) By [3], we have
$$\forall p\geq 1\quad \forall x,x'\in R\;
|\; S_{w_\ga,p}(x)-S_{w_\ga,p}(x')|\leq |x-x'|$$
Therefore,
\begin{align*}
\|\bS_{W,\bP}(v)-\bS_{W,\bP}(v')\|^2 
& = \sum_{\ga\in\Ga} 
|S_{w_\ga,p_\ga}(v_\ga)-S_{w_\ga,p_\ga}(v'_\ga)|^2\\
& =\sum_{i=1}^n \sum_{\ga\in\Ga_i} 
|S_{w_\ga,p_i}(v_\ga)-S_{w_\ga,p_i}(v'_\ga)|^2\leq \sum_{i=1}^n 
\sum_{\ga\in\Ga_i} |v_\ga-v'_\ga|^2\\
&=\sum_{\ga\in\Ga} |v_\ga-v'_\ga|^2=\|v-v'\|^2
\end{align*}
(D) 
$$|\!|\!|f^n|\!|\!|^{p_i}_{W_i,p_i} \leq |\!|\!|f^n|\!|\!|^{\bP}_{W,\bP} \leq 
\|K(f^n)-g\|^2+ 
|\!|\!|f^n|\!|\!|^{\bP}_{W,\bP}=\Phi_{W,\bP}(f^n)\leq \Phi_{W,\bP}(f^0)$$
For every $\ga\in \Ga_i$ we have
\begin{align*}
w_\ga^{\f{2-p_i}{p_i}} |f_\ga|^{2-p_i} &=(w_\ga|f_\ga|^{p_i})^{\f{2-p_i}{p_i}} 
\leq \left( \sum_{\ga\in\Ga_i} w_\ga|f_\ga|^{p_i}\right)^{\f{2-p_i}{p_i}} = 
|\!|\!|f|\!|\!|^{2-p_i}_{W_i,p_i}\\
&\Lo \sup_{\ga\in\Ga_i} \left[ w_\ga^{\f{2-p_i}{p_i}} |f_\ga|^{2-p_i}\right] 
\leq |\!|\!|f|\!|\!|^{2-p_i}_{W_i,p_i} \tag{$1$}
\end{align*}
For every $p_i\in [1,2]$,
\begin{align*}
\forall \ga\in \Ga_i \;\;\;
 c\leq w_\ga & \Lo c^{\f{2}{p_i}} \leq w_\ga^{\f{2}{p_i}} 
=w_\ga w_\ga^{\f{2-p_i}{p_i}} \\
&\Lo \f{1}{w_\ga} \leq \f{w_\ga^{\f{2-p_i}{p_i}}}{c^{\f{2}{p_i}}} \tag{$2$}
\end{align*}
We deduce from (1), (2) that
\begin{align*}
\sum_{\ga\in\Ga_i} |f_\ga|^2 &=\sum_{\ga\in\Ga_i} \f{1}{w_\ga} |f_\ga|^{2-p_i} 
w_\ga|f_\ga|^{p_i} \\
&\leq (\sup_{\ga\in\Ga_i} \f{1}{w_\ga} |f_\ga|^{2-p_i} )\sum_{\ga\in\Ga_i} 
w_\ga|f_\ga|^{p_i} \\
&\leq (\sup_{\ga\in\Ga_i} \f{w_\ga^{\f{2-p_i}{p_i}}}{c^{\f{2}{p_i}}} 
|f_\ga|^{2-p_i}) |\!|\!|f|\!|\!|^{p_i}_{W_i,p_i} \\
&\leq c^{-\f{2}{p_i}} |\!|\!|f|\!|\!|^{2-p_i}_{W_i,p_i} 
|\!|\!|f|\!|\!|^{p_i}_{W_i,p_i} 
= c^{-\f{2}{p_i}} |\!|\!|f|\!|\!|^2_{W_i,p_i} 
\end{align*}
Consequently, for every $i; 1\leq i\leq n$, we have
\begin{align*}
\sum_{\ga\in\Ga_i} |(f^n)_\ga|^2 & \leq c^{-\f{2}{p_i}} 
|\!|\!|f^n|\!|\!|^2_{W_i,p_i} \\
&\leq c^{-\f{2}{p_i}} \Phi_{W,\bP}(f^0)^{\f{2}{p_i}}:=M_i \\
\Lo \|f^n\|^2&=\sum_{\ga\in\Ga} |(f^n)_\ga|^2=\sum_{i=1}^n \sum_{\ga\in\Ga_i} 
|(f^n)_\ga|^2\leq \sum_{i=1}^n M_i\\
\Lo \exists M>0;& \forall n\in N\quad \|f^n\|\leq M.
\end{align*}
\end{proof}

We recall Lemma 2.9 and Lemma 2.10, from [8] and Appendix B in [3]
with minor modification.

\begin{lemma}\label{2.9}
Suppose for mapping $\bA$ from $\cH$ to $\cH$ and $v_0\in\cH$ we have 
following conditions 
\begin{itemize}
\item[(i)]  $\forall v,v'\in \cH$; $\|\bA(v)-\bA(v')\| \leq \|v-v'\|$ 
\item[(ii)] $\|\bA^{n+1}(v_0)-\bA^n(v_0)\|\lo 0$ \quad as\quad $n\lo\infty$
\item[(iii)] $\{v\in\cH| \bA(v)=v\}\neq \phi$.
\end{itemize}
Then 
$$\exists v^*\in {\cH}; \quad \bA(v^*)=v^*, A^n(v_0)\ovs{weakly}{\lo} v^*.$$
\end{lemma}

\begin{lemma}\label{2.10}
Suppose for mapping $\bA$ from $\cH$ to $\cH$ and $v_0\in\cH$ we have the
following conditions 
\begin{itemize}
\item[(i)]  $\forall v,v'\in \cH$; $\|\bA(v)-\bA(v')\| \leq \|v-v'\|$ 
\item[(ii)] $\|\bA^{n+1}(v_0)-\bA^n(v_0)\|\lo 0$ \quad as\quad $n\lo\infty$
\item[(iii)] $\exists 
\{A^{n_k}(v_0)\}_{k=1}^{+\infty}\subs\{A^n(v_0)\}_{n=1}^{+\infty}$, $\exists 
v'\in{\cH}$; $A^{n_k}(v_0)\ovs{weakly}{\lo} v'$
\end{itemize}
Then 
$$ A(v')= v'.$$
\end{lemma}
By Lemma2.9 and Lemma 2.10, we clearly have the following 
corollary:

\begin{corollary}\label{2.11}
Suppose for mapping $\bA$ from $\cH$ to $\cH$ and $v_0\in\cH$ we have the
following conditions 
\begin{itemize}
\item[(i)] $\forall v,v'\in \cH$; $\|\bA(v)-\bA(v')\|\leq \|v-v'\|$
\item[(ii)] $\|\bA^{n+1}(v_0)-\bA^n(v)\|\lo 0$\quad as\quad $n\lo\infty$
\item[(iii)] $\exists \{\bA^{n_k}(v_0)\}_{k=1}^\infty \subs 
\{\bA^n(v_0)\}_{n=1}^{+\infty}$, $\exists v'\in\cH; \bA^{n_k}(v_0) 
\ovs{weakly}{\lo} v'$.
\end{itemize}
Then 
$$\exists v^*\in{\cH};\quad \bA(v^*)=v^*, \bA^n(v_0) \ovs{weakly}{\lo} v^*.$$
\end{corollary}

We recall the following functional analysis lemma from [7]:
\begin{lemma}\label{2.12}
Every bounded sequence in a reflexive space has a weakly convergent 
subsequence.
\end{lemma}
We can now establish the following.

\begin{lemma}\label{2.13}
There exists $f^*\in\cH$ such that $\bT(f^*)=f^*$, $\bT^n(f^0)\ovs{weakly}{\lo}
f^*$.
\end{lemma}

\begin{proof}
Parts B,F of Lemma 2.8
 satisfy conditions (i), (ii) of Corollary 2.11. On the 
other hand, since $\cH$ is a reflexive space, Lemma 2.12
 and part D of Lemma 2.8 
satisfy condition (iii) of Corollary 2.11.
 Therefore Lemma 2.13 is proved, i.e., 
there exists $f^*\in\cH$ such that $\bT(f^*)=f^*$, $f^n=\bT^n(f^0) 
\ovs{weakly}{\lo} f^*$.
\end{proof}

In the last part of this subsection, we present the
 weak convergence theorem of 
sequence $\{f^n\}_{n=1}^{+\infty}$.

\begin{theorem}\label{2.14} 
{\bf (Weak Convergence)} Make the same assumptions as in Lemma 2.8, then we  have 

A) There exists $f^*\in\cH$ such that $f^n\ovs{weakly}{\lo} f^*$ and $f^*$ is 
the minimizer of $\Phi_{W,\bP}$.

B) If either there exist $i; 1\leq i\leq n$ such that $p_i>1$ or $N(K)=0$, 
then the minimizer of $\Phi_{W,\bP}$ is unique.
\end{theorem}

\begin{proof}
Part (A) is the direct result of Lemma 2.7, 2.13.
 For part (B), note that the following 
inequality becomes strict for one $p_i>1$:
$$\sum_{\ga\in\Ga_i} w_\ga \left| (\f{f_1+f_2}{2})_\ga\right|^{p_i} \leq 
\f{\sum_{\ga\in\Ga_i} w_\ga|(f_1)_\ga|^{p_i}+ \sum_{\ga\in\Ga_i} 
w_\ga|(f_2)_\ga|^{p_i}}{2}$$
Consequently
$$\sum_{i=1}^n |\!|\!|\f{f_1+f_2}{2}|\!|\!|^{p_i}_{W_i,p_i} < \f{\sum_{i=1}^n 
|\!|\!|f_1|\!|\!|^{p_i}_{W_i,p_i} +\sum_{i=1}^n 
|\!|\!|f_2|\!|\!|^{p_i}_{W_i,p_i}}{2}$$
Then, we have
$$|\!|\!| \f{f_1+f_2}{2} |\!|\!|^{\bP}_{W,\bP} < 
\f{|\!|\!|f_1|\!|\!|^{\bP}_{W,\bP}+|\!|\!|f_2|\!|\!|^{\bP}_{W,\bP}}{2}.$$
\end{proof}
\subsection{Strong Convergence of the $f^n$}
First, let us introduce the following shorthand notations:
$$f^*=w-\lim f^n,\quad u_n=f^n-f^*,\quad h=f^*+K^*(g-Kf^*)$$

\begin{theorem}\label{2.15}
{\bf (Strong Convergence)} Considering all assumptions of Lemma 2.8, we have 

A) $\|Ku^n\|\lo 0$ for $n\lo\infty$.

B) $\|\bS_{W,\bP}(h+u^n)-\bS_{W,\bP}(h)-u^n\|\lo 0$ for $n\lo\infty$.

C) If for some $a\in\cH$, and some sequence $(v^n)_{n\in\BN}$, 
$w-\lim_{n\rig\infty} v^n=0$ and $\lim_{n\rig\infty}\|\bS_{W,\bP}(a+v^n) 
-\bS_{W,\bP}(a)-v^n\|=0$ then $\|v^n\|\lo 0$ for $n\lo\infty$.

D) $\|u^n\|\lo 0$ for $n\lo\infty$.
\end{theorem}

\begin{proof}
Parts A,B have similar proofs to corresponding lemmas in [3]. 

C)

We know for every $i$, $p_i\geq 1$ and ${\bP}=\{p_1,\dots,p_n\}$.
Suppose $D_1=\{i| p_i>1\}$ and $D_2=\{i| p_i=1\}$. Consquently 
$$D_1\cup D_2=\{1,2,\dots,n\}$$
To prove this part, it is sufficient to consider the two distinct cases $i\in 
D_1$, $i\in D_2$, as has been done for $p>1$ and $p=1$, respectively, in [3]. 
We will have:
\begin{align*}
&\forall i\in D_1\quad \sum_{\ga\in\Ga_i} |v_\ga^n|^2\lo 0\quad as\quad 
n\lo\infty,\\
&\forall i\in D_2\quad \sum_{\ga\in\Ga_i} |v_\ga^n|^2\lo 0\quad as\quad 
n\lo\infty
\end{align*}
Then 
$$\|v^n\|^2=\sum_{\ga\in\Ga}|v^n_\ga|^2=\sum_{i=1}^n \sum_{\ga\in\Ga_i} 
|v_\ga^n|^2\lo 0\quad as \quad n\lo\infty$$
Therefore,
$$\|v^n\|\lo 0\quad as \quad n\lo\infty.$$
\end{proof}

\subsection{A Regularization Theorem}
The following Lemma has been estabilished in [3]:

\begin{lemma}\label{2.16}
Functions $S_{w,p}$ from $R$ to itself, defined in Lemma 2.1 satisfy 
$$|S_{w,p}(x)-x| \leq \f{wp}{2} |x|^{p-1}$$
\end{lemma}
To prove the regularization theorem, we will need the following lemma.

\begin{lemma}\label{2.17}
If the sequence of vectors $(v_k)_{k\in\BN}$ converges weakly in $\cH$ to 
$v$, and $\lim_{k\rig\infty} |\!|\!|v_k|\!|\!|_{W,\bP}^{\bP}=|\!|\!|v|\!|\!|
^{\bP}_{W,\bP}$, 
then $(v_k)_{k\in\BN}$ converges to $v$ in the $\cH$-norm, i.e.
$$\lim_{k\rig\infty} \|v-v_k\|=0.$$
\end{lemma}

\begin{proof}
It is a standard result that if $w-\lim_{k\rig\infty} v_k=v$, and 
$\lim_{k\rig\infty} \|v_k\|=\|v\|$, then 
$$\lim_{k\rig\infty} \|v-v_k\|^2 =\lim_{k\rig\infty} (\|v\|^2+ 
\|v_k\|^2-2<v,v_k>) = (\|v\|^2+\|v\|^2-2<v,v>)=0.$$
We thus need to prove only that $\lim_{k\rig\infty} \|v_k\|=\|v\|$. 

Since, the $v_k$ converge weakly, they are uniformly bounded. It follows that 
the $|v_{k,\ga}|=|<v_k,\var_\ga>|$ are bounded uniformly in $k$ and $\ga$ by 
some finite number $M$. For $a,b>0$ and $1\leq r\leq 2$, we have 
$$|a^r-b^r| \leq r|a-b| \Max\{a,b\}^{r-1}$$
Consequently 
$$\|v_{k,\ga}|^2-|v_\ga|^2 | = |(|v_{k,\ga}|^{p_\ga})^{\f{2}{p_\ga}} 
-(|v_\ga|^{p_\ga})^{\f{2}{p_\ga}}| \leq \f{2}{p_\ga} M^{2-p_\ga} 
\|v_{k,\ga}|^{p_\ga}-|v_\ga|^{p_\ga}|$$
Let $M':=\Max\{\f{2}{p_i} M^{2-p_i}; 1\leq i\leq n\}$ then 
$$\forall \ga\in\Ga \; \forall k\in N\; ||v_{k,\ga}|^2-|v_\ga|^2| \leq 
M'||v_{k,\ga}|^{p_\ga}-|v_\ga|^{p_\ga}|$$
Since the $v_k$ convergence weakly, we have
\begin{align*}
\forall& \ga\in\Ga; <v_k,\var_\ga> \lo <v,\var_\ga> \quad as \quad k\lo\infty\\
\Lo& \forall \ga\in\Ga ; |v_{k,\ga}|  \lo |v_\ga| \quad as \quad k\lo\infty
\end{align*}
Define now $u_{k,\ga}=\min(|v_{k,\ga}|,|v_\ga|)$. Clearly $\forall \ga\in\Ga$: 
$\lim_{k\rig\infty} u_{k,\ga}=|v_\ga|$; since $\sum_{\ga\in\Ga_i} 
w_\ga|v_\ga|^{p_i}<\infty$, it follows by the dominated convergence theorem 
that $\lim_{k\rig\infty} \sum_{\ga\in\Ga_i}w_\ga 
u^{p_i}_{k,\ga}=\sum_{\ga\in\Ga_i} w_\ga|v_\ga|^{p_i}$, consequently
\begin{align*}
\lim_{k\rig\infty} \sum_{\ga\in\Ga} w_\ga u_{k,\ga}^{p_\ga} &=
\lim_{k\rig\infty} \sum_{i=1}^n\sum_{\ga\in\Ga_i}w_\ga u_{k,\ga}^{p_i}\\
&=\sum_{i=1}^n \lim_{k\rig\infty}\sum_{\ga\in\Ga_i}w_\ga u_{k,\ga}^{p_i}\\
&=\sum_{i=1}^n\sum_{\ga\in\Ga_i} w_\ga|v_\ga|^{p_i} \\
&=\sum_{\ga\in\Ga} w_\ga |v_\ga|^{p_\ga} \tag{$6$}
\end{align*}
On the other hand, we have
\begin{align*}
| \|v_k\|^2-\|v\|^2|
 & = \left| \sum_{\ga\in\Ga} |v_{k,\ga}|^2-\sum_{\ga\in\Ga}|v_\ga|^2\right| \\
&\leq \sum_{\ga\in\Ga} | |v_{k,\ga}|^2-|v_\ga|^2| \\
&\leq M' \sum_{\ga\in\Ga} | |v_{k,\ga}|^{p_\ga} - |v_\ga|^{p_\ga}|\\
&\leq \f{M'}{c} \sum_{\ga\in\Ga} w_\ga | |v_{k,\ga}|^{p_\ga} - 
|v_\ga|^{p_\ga}|\\
&=\f{M'}{c}\sum_{\ga\in\Ga} w_\ga (|v_{k,\ga}|^{p_\ga} +|v_\ga|^{p_\ga} 
-2u_{k,\ga}^{p_\ga}) \\
&=\f{M'}{c} \left( \sum_{\ga\in\Ga} w_\ga |v_{k,\ga}|^{p_\ga} 
+\sum_{\ga\in\Ga} w_\ga | v_\ga|^{p_\ga} -2\sum_{\ga\in\Ga} w_\ga 
u_{k,\ga}^{p_\ga}\right).
\end{align*}
Since we have (6), the last expression tends to 0 as $k$ tends to $\infty$.
\end{proof}

The following existence lemma is necessary for providing Proposition 2.19 and 
Theorem 2.20.

\begin{lemma}\label{2.18}
Suppose $S=N(K)+f_0=\{f: K(f)=K(f_0)\}$ and assume that either there exists 
$j$ such that $p_j>1$ or $N(K)=\{0\}$. Then there is a
 unique minimal element 
with regard to $|\!|\!|\cdot|\!|\!|^{\bP}_{W,\bP}$ in $S$.
\end{lemma}

\begin{proof}
{\bf uniqueness:} If
 $N(K)=0$, then $S=\{f_0\}$. For the case where there exists $j$ such that 
$p_j>1$, suppose $f_1,f_2$ are minimal elements with regard to 
$|\!|\!|\cdot|\!|\!|^{\bP}_{W,\bP}$ in $S$ and $f_1\neq f_2$;
consequently $|\!|\!|f_1|\!|\!|^{\bP}_{W,\bP}=|\!|\!|f_2|\!|\!|_{W,\bP}^{\bP}$,
\begin{align*}
p_j>1 \Lo |\!|\!| \f{f_1+f_2}{2} |\!|\!|^{p_j}_{W_j,p_j}& < 
\f{|\!|\!|f_1|\!|\!|^{p_j}_{W_j,p_j}+|\!|\!|f_2|\!|\!|^{p_j}_{W_j,p_j}}{2}\\
\Lo|\!|\!|f_1|\!|\!|^{\bP}_{W,\bP} & \leq 
|\!|\!| \f{f_1+f_2}{2}|\!|\!|^{\bP}_{W,\bP} = 
|\!|\!| \f{f_1+f_2}{2}|\!|\!|^{p_1}_{W_1,p_1} +\dots+
|\!|\!| \f{f_1+f_2}{2}|\!|\!|^{p_n}_{W_n,p_n} \\
&< \f{|\!|\!|f_1|\!|\!|^{p_1}_{W_1,p_1}
 + |\!|\!|f_2|\!|\!|^{p_1}_{W_1,p_1}}{2} + \dots + 
\f{|\!|\!|f_1|\!|\!|^{p_n}_{W_n,p_n} + |\!|\!|f_2|\!|\!|^{p_n}_{W_n,p_n}}{2}\\
&=\f{|\!|\!|f_1|\!|\!|^{\bP}_{W,\bP} +|\!|\!|f_2|\!|\!|^{\bP}_{W,\bP}}{2} \\
&=\f{2|\!|\!| f_2|\!|\!|^{\bP}_{W,\bP}}{2}\\
&=|\!|\!| f_2|\!|\!|^{\bP}_{W,\bP}
\end{align*}
Therefore
$$|\!|\!|f_1|\!|\!|^{\bP}_{W,\bP} < |\!|\!|f_2|\!|\!|^{\bP}_{W,\bP}$$
which is a contradiction.\\[0.2cm]
{\bf Existence:}
Note that $|\!|\!|\cdot|\!|\!|^{\bP}_{W,\bP}$ is not a norm, but 
$|\!|\!|\cdot|\!|\!|_{W_i,p_i}$ 
are norms for $1\leq i\leq n$. Suppose $f_i$ is the element of minimum 
$|\!|\!|\cdot|\!|\!|_{W_i,p_i}$-norm in $S$.

Now, let $f^\dagger=\sum_{i=1}^n \sum_{\ga\in\Ga_i} <f_i,\var_\ga>\var_\ga$,
then we have
\begin{align*}
|\!|\!|f^\dagger |\!|\!|^{\bP}_{W,\bP} &=|\!|\!| 
f^\dagger|\!|\!|^{p_1}_{W_1,p_1} 
+\dots+|\!|\!|f^\dagger|\!|\!|^{p_n}_{W_n,p_n} \\
&=|\!|\!|f_1|\!|\!|^{p_1}_{W_1,p_1} +\dots+|\!|\!| f_n|\!|\!|^{p_n}_{W_n,p_n}
\end{align*}
because $<f^\dagger,\var_\ga>=<f_i,\var_\ga>$ for $\ga\in\Ga_i$. Consequently, 
$f^\dagger$ is a minimum with regard to $|\!|\!|\cdot|\!|\!|^{\bP}_{W,\bP}$ in 
$S$.
\end{proof}

\begin{proposition}\label{2.19}
Assume that $K$ is a bounded operator from $\cH$ to $\cH'$ with $\|K\|<1$,
$(\var_\ga)_{\ga\in\Ga}$ is an orthonormal basis for $\cH$, and 
$W=(w_\ga)_{\ga\in\Ga}$ is a sequence such that $\forall \ga\in\Ga w_\ga>c>0$. 
Let $\Ga=\Ga_1\cup\dots\cup \Ga_n, W_i=(w_\ga)_{\ga\in\Ga_i}$, 
$\bP=\{p_1,\dots,p_n\}$ such that  $1\leq p_i\leq 2$ for $1\leq i\leq n$. 
Suppose $g$ is an element of $\cH'$, $\al=(\al_1,\al_2,\dots,\al_n)$ such that 
$\al_i\geq 0$ for $1\leq i\leq n$, and that either there exists $j$ such 
that $p_j>1$ or $N(K)=\{0\}$. Define the functional $\Phi_{\al,W,\bP;g}$ on 
$\cH$ by 
$$\Phi_{\al,W,\bP;g}(f)=\|Kf-g\|^2+ \al_1|\!|\!|f|\!|\!|^{p_1}_{W_1,p_1}+ 
\al_2|\!|\!|f|\!|\!|^{p_2}_{W_2,p_2} +\dots+\al_n|\!|\!|f|\!|\!|^{p_n}_{W_n,p_n}
$$
Also suppose $f^*_{\al,W,\bP;g}$ is the minimizer of the functional 
$\Phi_{\al,W,\bP;g}, f_0\in\cH, S=N(K)+f_0=\{f|K(f)=K(f_0)\}$, 
and $f^\dagger$ is the unique minimal element  with regard to 
$|\!|\!|\cdot|\!|\!|^{\bP}_{W,\bP}$ in 
$S$. Let $\{\eps_t\}_{t=1}^\infty$ be a sequence of positive 
numbers convergent to 0
 and $\al(\eps_t)=(\al_1(\eps_t),\dots,\al_n(\eps_t))$ such that 
$\lim_{t\rig+\infty} \al_i(\eps_t)=0$, $\lim_{t\rig+\infty} 
\f{\eps_t^2}{\al_i(\eps_t)} =0$, $\lim_{t\rig+\infty} 
\f{\al_i(\eps_t)}{\al_j(\eps_t)}=1$ for every $i,j$; $1\leq i,j\leq n$.
Suppose $\{g_n\}_{n=1}^{+\infty} \subs\cH'$ is a sequence such that, for every 
$n$, $\|g_n-Kf_0\|<\eps_n$. 
Then we have 
$$\|f^*_{\al(\eps_n);g_n}- f^\dagger\| \lo 0\quad as \quad n\lo\infty.$$
\end{proposition}

\begin{proof}
At first, we prove (*),(**),(***). \\
(*) \quad $\exists M>0;\quad \forall t\in\BN\quad \|f^*_{\al(\eps_t);g_t}\|<M$.

In the proof of part D of Lemma 2.8, we had 
\begin{align*}
\forall i; & 1\leq i\leq n \quad \sum_{\ga\in\Ga_i} |f_\ga|^2 \leq 
c^{-\f{2}{p_i}} |\!|\!|f|\!|\!|^2_{W_i,p_i} \\
\Lo\forall i; & 1\leq i\leq n \quad \left(\sum_{\ga\in\Ga_i}
 |f_\ga|^2\right)^{\f{p_i}{2}} \leq 
\f{1}{c} \left(\sum_{\ga\in\Ga_i} w_\ga |f_\ga|^{p_i}\right)
\end{align*}
Since, $\{\f{\eps_t^2}{\al_i(\eps_t)} \}_{t=1}^{+\infty}$, 
$\{\f{\al_i(\eps_t)}{\al_j(\eps_t)}\}_{t=1}^{+\infty}$ are convergent 
sequences, we have
$$\exists M''>0; \;\forall t \; \left| \f{\eps_t^2}{\al_i(\eps_t)}\right|<M'', 
\quad \left| \f{\al_i(\eps_t)}{\al_j(\eps_t)}\right| <M''$$
Consequently,
\begin{align*}
\left( \sum_{\ga\in\Ga_i} |(f^*_{\al(\eps_t);g_t})_\ga|^2\right)^{\f{p_i}{2}} 
&\leq \f{1}{c} |\!|\!|f^*_{\al(\eps_t);g_t}|\!|\!|^{p_i}_{W_i,p_i} \\
&\leq \f{1}{c\al_i(\eps_t)} \Phi_{\al(\eps_t);g_t} (f^*_{\al(\eps_t);g_t}) 
\leq \f{1}{c\al_i(\eps_t)} \Phi_{\al(\eps_t);g_t} (f^\dagger) \\
&=\f{1}{c\al_i(\eps_t)} \left(\|Kf_0-g_t\|^2+
 \sum_{j=1}^n \al_j(\eps_t)|\!|\!| 
f^\dagger |\!|\!|^{p_j}_{W_j,p_j}\right) \\
&\leq \f{1}{c} \left( \f{\eps_t^2}{\al_i(\eps_t)} + \sum_{j=1}^n 
\f{\al_j(\eps_t)}{\al_i(\eps_t)} |\!|\!|f^\dagger 
|\!|\!|^{p_j}_{W_j,p_j}\right) \\
&\leq \f{1}{c} \left( M''+M'' \sum_{j=1}^n |\!|\!|f^\dagger 
|\!|\!|^{p_j}_{W_j,p_j} 
\right):=M'\tag{$7$}
\end{align*}
Then
\begin{align*}
& \sum_{\ga\in\Ga_i} |(f^*_{\al(\eps_t);g_t})_\ga|^2 \leq (M')^{\f{2}{p_i}} \\
&\Lo \|f^*_{\al(\eps_t);g_t}\|^2 =\sum_{i=1}^n \sum_{\ga\in\Ga_i} 
|(f^*_{\al(\eps_t);g_t})_\ga|^2\leq \sum_{i=1}^n (M')^{\f{2}{p_i}} \\
&\Lo \exists M>0; \forall t\in\BN\quad \|f^*_{\al(\eps_t);g_t}\|<M.
\end{align*}
Furthermore, we have
\begin{align*}
|\!|\!|f^*_{\al(\eps_t);g_t} |\!|\!|^{\bP}_{W,\bP} 
&\leq \f{1}{\min \{\al_i(\eps_t)|1\leq i\leq n\}} \Phi_{\al(\eps_t);g_t} 
(f^*_{\al(\eps_t);g_t}) \\
&\leq \f{1}{\min\{\al_i(\eps_t) | 1\leq i\leq n\}} \Phi_{\al(\eps_t);g_t} 
(f^\dagger) \\
&=\f{1}{\min\{\al_i(\eps_t)| 1\leq i\leq n\}}  (\|Kf_0-g_t\|^2+ \sum_{j=1}^n 
\al_j(\eps_t)|\!|\!| f^\dagger|\!|\!|^{p_j}_{W_j,p_j}) \\
&\leq \f{\eps_t^2}{\min\{\al_i(\eps_t)| 1\leq i\leq n\}} +\sum_{j=1}^n 
\f{\al_j(\eps_t)}{\min\{\al_i(\eps_t)| 1\leq i\leq n\}} 
|\!|\!|f^\dagger|\!|\!|^{p_j}_{W_j,p_j}
\end{align*}
Then 
\begin{equation*}
\lim_{t\rig\infty} |\!|\!|f^*_{\al(\eps_t);g_t} |\!|\!|^{\bP}_{W,\bP} \leq 0+ 
\sum_{j=1}^n |\!|\!|f^\dagger |\!|\!|^{p_j}_{W_j,p_j}= |\!|\!|f^\dagger 
|\!|\!|^{\bP}_{W,\bP} 
\tag{$8$}
\end{equation*}
(**)\quad 
If $\{\til{f}\}_{k=1}^\infty$ is subsequence of 
$\{f^*_{\al(\eps_t);g_t}\}_{t=1}^{+\infty}$ and $\til{f}\in\cH$,
$\til{f}_k\ovs{weakly}{\lo} \til{f}$ and 
$\{|\!|\!|\til{f}_k|\!|\!|^{\bP}_{W,\bP}\}_{k\in\BN}$ are convergent, then we 
have 
$$|\!|\!|\til{f}_k|\!|\!|^{\bP}_{W,\bP} \lo 
|\!|\!|\til{f}|\!|\!|^{\bP}_{W,\bP}, 
f^\dagger=\til{f}$$

At first, we introduce the following new notations
$$\til{f}_k\in \{f^*_{\al(\eps_t);g_t}\}_{t\in\BN} \Lo \exists m; 
\til{f}_k=f^*_{\al(\eps_m);g_m}$$
Let $\til{\al}_k:=\al(\eps_m)=(\al_1(\eps_m),\dots,\al_n(\eps_m))$, 
$\til{g}_k:=g_m,\til{e}_k:=\til{g}_k-K(f_0),(\til{\al}_k)_\ga:=\al_i(\eps_m)$ 
for $\ga\in\Ga_i$. By these notations, we have
$$\|\til{e}_k\| <\til{\eps}_k, \til{f}_k \; \text{is the minimizer of the 
functional}\; \Phi_{\til{\al}_k;\til{g}_k}.$$

Since,  in the assumptions we have:
 either there exists $p_i>1$ or $N(K)=0$, it 
follows from 
Theorem 2.14 that $\til{f}_k$ is the unique minimizer of 
$\Phi_{\til{\al}_k,\til{g}_k}$. On the other hand, by Lemma 2.7, 2.13, $\til{f}_k$ 
is a fixed point for $\til{\bT}_k$, i.e. $\til{\bT}_k(\til{f}_k)=\til{f}_k$.
Consequently 
$$\til{f}_k=\til{\bT}_k(\til{f}_k) =\bS_{\til{\al}_k,W,\bP} 
(\til{f}_k+K^*(\til{g}_k-K(\til{f}_k)))$$
Let $\til{f}_\ga:=<\til{f},\var_\ga>$, 
$\til{h}_k:=\til{f}_k+K^*(\til{g}_k-K\til{f}_k)$.

With these new notations, we have 
\begin{align*}
(\til{f}_k)_\ga &=S_{(\til{\al}_k)_\ga w_\ga,p_\ga}((\til{h}_k)_\ga),\\
\til{h}_k:&=\til{f}_k+K^*(\til{g}_k-K(\til{f}_k)) 
=\til{f}_k+K^*K(f_0-\til{f}_k)+K^*(\til{e}_k) \tag{$9$}
\end{align*}
Moreover,  by (*), $2\|\til{f}_k\|+\|f_0\| \leq 2M+\|f_0\| :=M'''$ and since 
$\|K\|\leq 1$, we have 
\begin{align*}
\|\til{h}_k\| & \leq \|\til{f}_k\| +\|f_0-\til{f}_k\| +\|\til{e}_k\| \\
&\leq 2\|\til{f}_k\| +\|f_0\| +\til{\eps}_k\\
&\leq M'''+\til{\eps}_k
\end{align*}
Therefore, for every $\ga\in\Ga$, 
\begin{align*}
|S_{(\til{\al}_k)_{\ga w_\ga,p_\ga}} ((\til{h}_k)_\ga)-(\til{h}_k)_\ga| 
& \leq  p_\ga w_\ga(\til{\al}_k)_\ga \f{|(\til{h}_k)_\ga|^{p_\ga-1}}{2}\\
& \leq  p_\ga w_\ga(\til{\al}_k)_\ga \f{(M'''+\til{\eps}_k)^{p_\ga-1}}{2}
\end{align*}
and since $\til{\eps}_k\lo 0$, $(\til{\al}_k)_\ga\lo 0$ as $k\lo\infty$, then, 
\begin{equation*}
|S_{(\til{\al}_k)_\ga w_\ga,p_\ga} ((\til{h}_k)_\ga) -(\til{h}_k)_\ga|\lo 0
\; as\; k\lo\infty \tag{$10$}
\end{equation*}
 Since $\til{f}_k\ovs{w}{\lo} \til{f}$, we have, for every $\ga\in\Ga$, 
$(\til{f}_k)_\ga\lo (\til{f})_\ga$ as $k\lo\infty$. This implies that, when 
$k\lo\infty$
$$|[K^*K(f_0-\til{f}_k)]_\ga -[K^*K(f^\dagger -\til{f})]_\ga| \lo 
[K^*K(f_0-f^\dagger)]_\ga=0$$
and hence 
\begin{equation*}
[K^*K(f_0-\til{f}_k)]_\ga\lo [K^*K(f^\dagger-\til{f})]_\ga \tag{$11$}
\end{equation*}
On the other hand, by the following inequality
$$(K^*\til{e}_k)_\ga \leq \|K^*\til{e}_k\| \leq \|\til{e}_k\|<\til{\eps}_k$$
we have
\begin{equation*}
(K^*\til{e}_k)_\ga\lo 0 \quad as \quad k\lo\infty \tag{$12$}
\end{equation*}
By (9), (10), (11) and (12), we have
\begin{align*}
\til{f}_\ga &=\lim_{k\rig\infty} (\til{f}_k)_\ga =\lim_{k\rig\infty} 
S_{(\til{\al}_k)_\ga w_\ga,p_\ga} [(\til{h}_k)_\ga] \\
&=\lim_{k\rig\infty} S_{(\til{\al}_k)_\ga w_\ga,p_\ga}[(\til{h}_k)_\ga]-
(\til{h}_k)_\ga+ \lim_{k\rig\infty}(\til{h}_k)_\ga\\
&=\lim_{k\rig\infty} (\til{h}_k)_\ga\\
&=\lim_{k\rig\infty} (\til{f}_k)_\ga 
+[K^*K(f_0-\til{f}_k)]_\ga+(K^*\til{e}_k)_\ga\\
&=\til{f}_\ga+[K^*K(f^\dagger-\til{f})]_\ga
\end{align*}
Consequently 
\begin{align*}
& \forall \ga\in\Ga; \quad 
\til{f}_\ga=\til{f}_\ga+[K^*K(f^\dagger-\til{f})]_\ga\\
\Lo & \forall \ga\in\Ga;\quad [K^*K(f^\dagger-\til{f})]_\ga=0\\
\Lo & K^*K(f^\dagger-\til{f})=0\\
\Lo & f^\dagger-\til{f}\in N(K)\\
\Lo & K(\til{f})=K(f^\dagger) \\
\Lo & \til{f}\in S\\
\Lo & |\!|\!| f^\dagger |\!|\!|^{\bP}_{W,\bP} \leq 
|\!|\!|\til{f}|\!|\!|^{\bP}_{W,\bP} 
\tag{$13$}
\end{align*}
 By Fatou's lemma and (8),(13), we have
\begin{align*}
|\!|\!|\til{f}|\!|\!|^{\bP}_{W,\bP} &=\sum_{\ga\in\Ga} w_\ga 
|\til{f}_\ga|^{p_\ga} 
\\
&=\sum_{\ga\in\Ga} w_\ga (\lim_{k\rig\infty} |(\til{f}_k)_\ga|^{p_\ga}) \\
&\leq \lim_{k\rig\infty} \sup \sum_{\ga\in\Ga} w_\ga |(\til{f}_k)_\ga|^{p_\ga} 
\\
&=\lim_{k\rig\infty} |\!|\!|\til{f}_k|\!|\!|^{\bP}_{W,\bP} \leq 
|\!|\!|f^\dagger|\!|\!|^{\bP}_{W,\bP} \\
&\leq |\!|\!|\til{f}|\!|\!|^{\bP}_{W,\bP}
\end{align*}
This implies that 
$$\lim_{k\rig\infty} |\!|\!|\til{f}_k|\!|\!|^{\bP}_{W,\bP} 
=|\!|\!|f^\dagger|\!|\!|^{\bP}_{W,\bP} =|\!|\!|\til{f}|\!|\!|^{\bP}_{W,\bP}$$
Since $f^\dagger$ is the unique minimal element with regard to 
$|\!|\!|\cdot|\!|\!|^{\bP}_{W,\bP}$ in $S$, it follows that 
$\til{f}=f^\dagger$. \\[0.2cm]
(***)\quad Let $\cH$ be a Hilbert space and $\{x_n\}_{n\in\BN}$ be a bounded 
sequence in $\cH$ such that $\{|\!|\!|x_n|\!|\!|^{\bP}_{W,\bP}\}_{n\in\BN}$ is 
bounded, too. Pick $x_0\in\cH$. Also suppose we know $x_{n_k}\ovs{w}{\lo} x_0$ 
is true  for every 
$\{x_{n_k}\}_{k\in\BN}$, a weakly  convergent subsequence of 
$\{x_n\}_{n\in\BN}$,  where 
 $\{|\!|\!|x_{n_k}|\!|\!|^{\bP}_{W,\bP}\}_{k\in\BN}$ is 
convergent. Then
$$x_n\ovs{w}{\lo} x_0.$$

First, note that for $\{a_n\}_{n=1}^{+\infty}\subs\BR$, we have $a_n\lo a$ 
if and only if 
$$\forall \{a_{n_k}\}_{k=1}^{+\infty} \subs\{a_n\}_{n=1}^{+\infty} \exists 
\{a_{n_{k_m}}\}_{m=1}^{+\infty} \subs \{a_{n_k}\}^{+\infty}_{k=1}; 
a_{n_{k_m}}\lo a_0$$
Also, we have
$$\forall f\in X^*; f(x_n)\lo  f(x_0)\LO x_n\ovs{weakly}{\lo} x_0$$
Pick arbitrary $f\in X^*$ and suppose $\{f(x_{n_k})\}_{k=1}^{+\infty}$ is a 
subsequence of $\{f(x_n)\}_{n=1}^{+\infty}$. Since 
$\{|\!|\!| x_{n_k}|\!|\!|^{\bP}_{W,\bP}\}_{k=1}^{+\infty}$,  and 
$\{x_{n_k}\}_{k=1}^{+\infty}$ are bounded, and also  by  Lemma 2.12,
 we conclude that 
there is $\{x_{n_{k_m}}\}_{m=1}^{+\infty}$, a subsequence of 
$\{x_{n_k}\}_{k=1}^{+\infty}$, such that $\{x_{n_{k_m}}\}_{m=1}^{+\infty}$ is 
weakly convergent and $\{|\!|\!|x_{n_{k_m}}|\!|\!|^{\bP}_{W,\bP}\}^{+\infty}_{m=
1}$ is convergent.

By the assumptions,
 $\{x_{n_{k_m}}\}_{m=1}^{+\infty}$ converges weakly to $x_0$, 
i.e. 
$$x_{n_{k_m}} \ovs{w}{\lo} x_0$$
Hence,
$$f(x_{n_{k_m}}) \lo f(x_0) \quad as \quad m\lo\infty$$
Then, we conclude
$$f(x_n)\lo f(x_0)$$
Consequently,
$$x_n\ovs{weakly}{\lo} x_0.$$

We have the following  inequality in (7):
$$\forall n\in\BN \quad |\!|\!|f^*_{\mu_n,g_n}|\!|\!|^{p_i}_{W_i,p_i} \leq 
M''+M'' 
\sum_{j=1}^n |\!|\!|f^\dagger|\!|\!|^{p_j}_{W_j,p_j}$$
This means $\{|\!|\!|f^*_{\mu_n,g_n}|\!|\!|^{\bP}_{W,\bP}\}_{n=1}^{+\infty}$
is bounded; this fact together with (*), (**)  provide the 
 assumptions of (***), for 
$\{f^*_{\mu_n,g_n}\}_{n=1}^{+\infty}$ and $f^\dagger\in\cH$, then by (***) 
we conclude:
\begin{equation*}
f^*_{\mu_n,g_n}\ovs{w}{\lo} f^\dagger \tag{$14$}
\end{equation*}
Suppose $\{|\!|\!|\til{f}_k|\!|\!|^{\bP}_{W,\bP}\}_{n=1}^{+\infty}$
is an arbitrary subsequence of 
$\{|\!|\!|f^*_{\mu_n,g_n}|\!|\!|^{\bP}_{W,\bP}\}_{n=1}^{+\infty}$. Since
 $\{|\!|\!|\til{f}_k|\!|\!|^{\bP}_{W,\bP}\}_{n=1}^{+\infty}$ is bounded, there 
is a convergent subsequence
 $\{|\!|\!|\til{f}_{k_m}|\!|\!|^{\bP}_{W,\bP}\}_{m=1}^{+\infty}$
of it, and then by (14), we have
$$\til{f}_{k_m}\ovs{w}{\lo} f^\dagger$$
So, by (**); we conclude:
$$|\!|\!|\til{f}_{k_m}|\!|\!|^{\bP}_{W,\bP} \lo 
|\!|\!|f^\dagger|\!|\!|^{\bP}_{W,\bP} \quad as \quad m\lo+\infty$$
And since
 $\{|\!|\!|\til{f}_k|\!|\!|^{\bP}_{W,\bP}\}_{k=1}^{+\infty}$ is an arbitrary 
subsequence 
of $|\!|\!|f^*_{\mu_n,g_n}|\!|\!|^{\bP}_{W,\bP}$, we conclude:
\begin{equation*}
|\!|\!|f^*_{\mu_n,g_n}|\!|\!|^{\bP}_{W,\bP} \lo 
|\!|\!|f^\dagger|\!|\!|^{\bP}_{W,\bP} \quad as 
\quad n\lo\infty\tag{$15$}
\end{equation*}
Finally, by (14), (15) and lemma 2.17,
$$\|f^*_{\mu_n,g_n}-f^\dagger\|\lo 0 \quad as\quad n\lo\infty.$$
\end{proof}

The following regularization theorem is our major goal in this subsection.

\begin{theorem}\label{2.20}
 Assume that $K$ is a bounded operator from $\cH$ to $\cH'$ 
with $\|K\|<1$,  $\{\var_\ga\}_{\ga\in\Ga}$  is an orthonormal basis for 
$\cH$, and $W=(w_\ga)_{\ga\in\Ga}$ is a sequence such that $\forall \ga\in\Ga$ 
$w_\ga>c>0$. Let $\Ga=\Ga_1\cup\Ga_2\cup\dots\cup \Ga_n$, 
$W_i=(w_\ga)_{\ga\in\Ga_i}$, $\bP=\{p_1,p_2,\dots,p_n\}$ such that $1\leq 
p_i\leq 2$ for $1\leq i\leq n$. Suppose that $g$ is an element of $\cH'$, 
$\al=(\al_1,\al_2,\dots,\al_n)$ such that $\al_i\geq 0$ for $1\leq i\leq n$, 
and that either there exists $j$ such that $p_j>1$ or $N(K)=\{0\}$. Define 
the functional $\Phi_{\al,W,\bP;g}$ on $\cH$ by 
$\Phi_{\al,W,\bP;g}(f)=\|Kf-g\|^2+\al_1|\!|\!|f|\!|\!|^{p_1}_{W_1,p_1}+\al_2 
|\!|\!|f|\!|\!|^{p_2}_{W_2,p_2}+\dots+\al_n |\!|\!|f|\!|\!|^{p_n}_{W_n,p_n}$.
Also assume
 $f^*_{\al,W,\bP;g}$ is the minimizer of the functional $\Phi_{\al,W,\bP;g}$.

Let $\al(\eps)=(\al_1(\eps),\al_2(\eps),\dots,\al_n(\eps))$ such that 
$\lim_{\eps\lo 0}\al_i(\eps)= 0$, $\lim_{\eps\rig 0}\f{\eps^2}{\al_i(\eps)} 
=0$, $\lim_{\eps\rig 0}\f{\al_i(\eps)}{\al_j(\eps)}=1$ for every $i,j; 1\leq 
i,j\leq n$. Then we have, for any $f_0\in\cH$,
$$\lim_{\eps\rig 0} \left[ \sup_{\|g-Kf_0\|<\eps} 
\|f^*_{\al(\eps),W,\bP;g}-f^\dagger\|\right]=0$$
where $f^\dagger$ is the unique minimal element with regard to 
$|\!|\!|\cdot|\!|\!|^{\bP}_{W,\bP}$ in $S=N(K)+f_0=\{f;K(f)=K(f_0)\}$ as 
stated  in Lemma 2.18.
\end{theorem}

\begin{proof}
Let $H(\eps):=\sup\{\|f^*_{\al(\eps),W,\bP;g}-f^\dagger\| |g\in\cH', 
\|g-Kf_0\|<\eps\}.$

We should establish $\lim_{\eps\rig 0} H(\eps)=0$. If $\lim_{\eps\rig 0} 
H(\eps)\neq 0$,
 then there is a sequence $\{\eps_n\}_{n=1}^{+\infty}$ such that 
$\eps_n\lo 0$ and $\{H(\eps_n)\}_{n=1}^{+\infty}$ is not convergent to 0. 
consequently,
\begin{align*}
&\exists
 \del_0>0; \exists \{H(\eps_{n_k})\}_{k=1}^{+\infty}; \forall k\in\BN\; 
H(\eps_{n_k})\geq \del_0\\
\Lo&\exists
 \del_0>0; \exists \{H(\eps_{n_k})\}_{n=1}^{+\infty}; \forall k\in\BN\; 
\exists g_{n_k}\in\cH'; \|g_{n_k}-Kf_0\|\leq 
\eps_{n_{k}}, \|f^*_{\al(\eps_{n_k});g_{n_k}}-f^\dagger\|\geq \del_0.
\end{align*}
This is a contradiction, because
by  Lemma 2.19 for $\{\eps_{n_k}\}_{k=1}^{+\infty}$, 
$\{g_{n_k}\}_{k=1}^{+\infty}$, we will have
$$\|f^*_{\al(\eps_{n_k});g_{n_k}} -f^\dagger\|\lo 0.$$
\end{proof}

\section{Outcomes in solving many minimization problems}
In this section we present, in three subsections, some applications of the 
generalization provided in section 2. 
\subsection{Multi-frame representations in linear inverse problems with mixed 
multi-constraints}
In this subsection, we prove theorems, from G. Teschke [9], as results 
from theorems in the previous
 section. Furthermore, we present a regularization 
theorem and prove it. We begin with the following simple lemma.

\begin{lemma}\label{3.1}
A) Suppose $\cH_1,\cH_2,\dots,\cH_n$ are Hilbert space with inner products 
$<.,.>_{\cH_i}$ for $\cH_i$. Then $\cH_1\ti\cH_2\ti\dots\ti \cH_n$ is a 
Hilbert space  with 
the inner product 
$\langle f,g\rangle_{\cH}=\langle f_1,h_1\rangle_{\cH_1}+ \dots+\langle 
f_n,h_n\rangle_{\cH_n}$ for $f=(f_1,\dots,f_n)$, $g=(g_1,\dots,g_n)$ such that 
$f_i,g_i\in\cH_i$, and with componential addition and componential scalar 
multiplication.

B) If $\{a_\la^i\}_{\la\in\La_i}$ is an orthonormal basis for $\cH_i$, then 
$\sA=\{A_\la^j| A_\la^j=(0,\dots,a_\la^j,0,\dots,0);\la\in\La_j, 1\leq j\leq 
n\}$ is an orthonormal basis for $\cH_1\ti\cH_2\ti\dots\ti \cH_n$.
\end{lemma}

\begin{remark}\label{3.2}
Suppose $\cH=\cH_1\ti\cH_2\ti\dots\ti\cH_n$ is a Hilbert space as 
identified in Lemma \ref{3.1} and $K:\cH\lo \cH'$ is a bounded linear operator.

We can have all the
 results in section 2 for $K$, $\cH$. By assuming $\Ga_i=\La_i$ 
and $\sA=\sA_1\cup\dots\cup \sA_n=\{A^1_\la\}_{\la\in\La_1} \cup\dots\cup 
\{A^n_\la\}_{\la\in\La_n}$ for $f=(f^1,\dots,f^n)\in\cH$, we have
\begin{align*}
|\!|\!|f|\!|\!|^{\bP}_{W,\bP} &=|\!|\!|f|\!|\!|^{p_1}_{W_1,p_1}
+\dots+|\!|\!|f|\!|\!|^{p_n}_{W_n,p_n}\\
&=\sum_{\la\in\La_1} w_\la| <(f,A_\la^1)>_{\cH} 
|^{p_1}+\dots+\sum_{\la\in\La_n} w_\la|<f,A^n_\la)>_{\cH}|^{p_n} \\
&=\sum_{\la\in\La_1} w_\la |<(f^1,\dots,f^n), 
(a^1_\la,0,\dots,0)>_{\cH}|^{p_1}+\dots\\
&\;\; + \sum_{\la\in\La_n} w_\la| 
<(f^1,\dots,f^n),(0,\dots,0,a^n_\la)>_{\cH}|^{p_n}\\
&=\sum_{\la\in\La_1}w_\la|<f^1,a^1_\la>_{\cH_1}|^{p_1}+ 
\dots+\sum_{\la\in\La_n} w_\la|<f^n,a^n_\la>_{\cH_n}|^{p_n}
\end{align*}
Also for $h=(h^1,\dots,h^n)$
\begin{align*}
\bS_{W,\bP}(h) &=\sum_{i=1}^n (\sum_{\la\in\La_i} 
S_{w_\la,p_\la}(<h,A^i_\la>_{\cH})A^i_\la) \\
&=\sum_{i=1}^n (\sum_{\la\in\La_i} 
S_{w_\la,p_i}(<h^i,a_\la^i>_{\cH_i})A_\la^i)\\
&=\sum_{i=1}^n (\sum_{\la\in\La_i} (0,\dots,S_{w_\la,p_i} 
(<h^i,a^i_\la>_{\cH_i})a^i_{\la},\dots,0))\\
&=\sum_{i=1}^n (0,\dots,\sum_{\la\in\La_i} 
S_{w_\la,p_i}(<h^i,a_\la^i>_{\cH_i})a_\la^i,\dots,0) \\ 
 &=(\sum_{\la\in\La_1} S_{w_\la,p_1} 
(<h^1,a^1_\la>_{\cH_1})a^1_\la,\dots,\sum_{\la\in\La_n} 
S_{w_\la,p_n}(<h^n,a_\la^n>_{\cH_n})a^n_\la).
\end{align*}
\end{remark}

\begin{lemma}\label{3.3}
A) If $\La$ is a countable set, then the following set is a Hilbert space
$$l_2=\{\{a_\la\}_{\la\in\La} | a_\la\in \BC, 
\sum_{\la\in\La}|a_\la|^2<\infty\}$$
considering $<a,b>=\sum_{\la\in\La}a_\la \bar{b}_\la$ for $a,b\in l_2$ 
and $a=\{a_\la\}_{\la\in\La}$, $b=\{b_\la\}_{\la\in\La}$. 

B) $E=\{e_\ga|e_\ga=\{a_\la\}_{\la\in\La}; a_\ga=1\; \text{for}\; \la=\ga, 
\forall\la\neq \ga\; a_\la=0\}$ is an orthonormal basis for $l_2$.
\end{lemma}

The following lemma is a simple result of Lemma \ref{3.1}, \ref{3.3}.

\begin{lemma}\label{3.4}
Suppose $\La_1,\La_2,\dots,\La_n$ are countable sets and $l_2^i$ for 
$i=1,\dots,n$ are Hilbert spaces as follows 
$$l_2^i=\{\{a_\la\}_{\la\in\La_i} | a_\la\in\BC, \sum_{\la\in\La_i} 
|a_\la|^2<\infty\}$$
Then, we have

A) $(l_2)^n:=l_2^1\ti\dots\ti l_2^n$ with the following inner product 
$$\langle f,h\rangle=\langle f_1,h_1\rangle_{l_2^1}+\dots+\langle 
f_n,h_n\rangle_{l_2^n}$$
is a Hilbert space.

B) If $\sA=\{E_\ga^j| E_\ga^j=(0,\dots,e^j_\ga,\dots,0), 1\leq j\leq n, 
\ga\in\La_j\}$ such that $e_\ga^j=\{a_\la\}_{\la\in\La_j}$, $a_\la=0$ for 
$\la\neq\ga$, $a_\la=1$ for $\la=\ga$, then $\sA$ is an orthonormal basis for 
$(l_2)^n$.
\end{lemma}

\begin{lemma}\label{3.5}
Suppose $\cH,\cH'$ are Hilbert spaces, 
$\{\phi_\la^1\}_{\la\in\La_1},\dots,\{\phi_\la^n\}_{\la\in\La_n}$ are frame
in $\cH$, $F_i:\cH\lo l_2^i$ are frame operators,
$F^*_i:l_2^i\lo\cH$ are adjoint operators
 of $F_i$, and also $A:\cH\lo \cH'$ is 
a bounded linear operator. Defined the function
 $K_A:(l_2)^n\lo \cH'$ by
$$K(v^1,\dots,v^n)=\sum_{i=1}^n A(F^*_i(v^i)).$$
Then $K_A$ is a bounded linear operator.
\end{lemma}

\begin{proof}
Since $F^*_i$, A are linear operators, $AF^*_i$ are linear operators,
 for $i=1,\dots,n$, and consequently 
$$K_A=\sum_{i=1}^n AF^*_i$$
is a linear operator, too.

After that, we indicate $K_A^*$ is bounded, as defined below,
$$K_A^*:\cH'\lo (l_2)^n, \quad K^*_A(g)=(F_1A^*(g),\dots,F_nA^*(g))$$
and is an adjoint operator for $K_A$, because for $h\in\cH'$  and 
$f=(f^1,\dots,f^n)\in (l_2)^n$, we have 
\begin{align*}
\langle K_A(f),h\rangle_{\cH'} &=
\sum_{i=1}^n \langle AF^*_i(f^i),h\rangle_{\cH'} \\
&=\sum_{i=1}^n \langle f^i,F_iA^*(h)\rangle_{l_2^i}\\
&=\langle f,(F_1A^*(h),\dots,F_nA^*(h))\rangle_{(l_2)^n} \\
&=\langle f,K^*_A(h)\rangle_{(l_2)^n}
\end{align*}
Then, for $f\in\cH'$, we have
\begin{align*}
\|K^*_A(f)\|^2_{(l_2)^n} &=\|F_1A^*(f)\|^2_{l_2^1}+ \dots+ 
\|F_nA^*(f)\|^2_{l_2^n}\\
&\leq \|F_1\|^2 \|A\|^2 \|f\|_{\cH'}^2+\dots+\|F_n\|^2 \|A\|^2 \|f\|^2_{\cH'}.
\end{align*}
By assuming $B_i$ to be an upper bound for $F_i$ and $\til{c}$ 
to be a bound for $A$, we have
$$\|K_A\|=\|K^*_A\|< \til{c} \sqrt{B_1+\dots+B_n}:=c.$$
\end{proof}

In Lemma \ref{3.5}, if was 
 proved that $\|K_A\|<c$; for simplicity, we will restrict ourselves to 
case $c=1$, without loss of generality, since $K_A$ can always be renormalized.

\begin{remark}\label{3.6}
Now, we utilize the results
 in Remark 3.2 for $\cH=(l_2)^n$, $\cH_i=l_2^i$, 
$A_i=\{E_\la^i\}_{\la\in\La_i}$ as identified in Lemma \ref{3.4}.
By assuming 
$W_i=\{\al_i w_i\}_{\la\in\La_i}$ such that $\al_i,w_i\in\BR$ for 
$i=1,\dots,n$, $f^i\in l_2^i$, $f^i=\{f^i_\la\}_{\la\in\La_i}$, 
$f=(f^1,\dots,f^n)\in(l_2)^n$,
$$|f^i|_{p_i,w_i}=\sum_{\la\in\La_i} w_i|f_\la^i|^{p_i}, \quad 
|\!|\!|f|\!|\!|=(|f^1|_{p_1,w_1},\dots,|f^n|_{p_n,w_n}),$$
$\al=(\al_1,\al_2,\dots,\al_n)$, we have
\begin{align*}
|\!|\!|f|\!|\!|^{\bP}_{W,\bP} &=\sum_{\la\in\La_1} \al_1 
w_1|<f^1,e_\la^1>_{l_2^1} 
|^{p_1}+\dots+\sum_{\la\in\La_n} \al_n w_n|<f^n, e^n_\la>_{l_2^n}|^{p_n} \\
&=(\al_1,\dots,\al_n). (\sum_{\la\in\La_1} w_1|<f^1,e^1_\la>_{l_2^1} |^{p_1} 
,\dots, \sum_{\la\in\La_n} w_n| <f^n,e_\la^n>_{l_2^n}|^{p_n})\\
&=(\al_1,\dots,\al_n). (\sum_{\la\in\La_1} w_1|f_\la^1|^{p_1} 
,\dots,\sum_{\la\in\La_n} w_n|f_\la^n|^{p_n}) \\
&=(\al_1,\dots,\al_n). (|f^1|_{p_1,w_1} ,\dots,|f^n|_{p_n,w_n}) \\
&=\al.|\!|\!|f|\!|\!|
\end{align*}
Therefore
$$|\!|\!|f|\!|\!|^{\bP}_{W,\bP} =\al.|\!|\!|f|\!|\!|$$
Also, we have
\begin{align*}
\bS_{W,\bP}(h) &=\left( \sum_{\la\in\La_1} S_{w_\la,p_1} (\langle 
h^1,e^1_\la\rangle_{l_2^1})e_\la^1,\dots, \sum_{\la\in\La_n} S_{w_\la,p_n} 
(\langle h^n,e_\la^n\rangle_{l_2^n}) e_\la^n\right) \\
&=\left( \sum_{\la\in\La_1} S_{w_1\al_1,p_1}(h^1_\la) e_\la^1,\dots, 
\sum_{\la\in\La_n} S_{w_n \al_n,p_n}(h^n_\la)e_\la^n\right) \\
&=\left( \{S_{w_1\al_1,p_1} (h^1_\la)\}_{\la\in\La_1} ,\dots, 
\{S_{w_n\al_n,p_n} (h_\la^n)\}_{\la\in\La_n}\right).
\end{align*}
\end{remark}

Remark 3.6
 and Lemma 3.5 in this section, together
 with Propositions 2.2, 2.19
 and Theorem 2.15 
in the previous section, give proofs of following theorems.

\begin{theorem}\label{3.7}
Suppose $K_A:(l_2)^n\lo \cH'$ is as identified in Lemma 3.5 and 
$\bP=\{p_1,p_2,\dots,p_n\}$ such that $1\leq p_i\leq 2$ for $i=1,\dots,n$.

For $\al=(\al_1,\dots,\al_n)\in R^n$, $a\in (l_2)^n$
and $g\in\cH'$, define  functionals $\Phi,\Phi^{S\cup R}$ on 
$(l_2)^n$ by 
\begin{align*}
\Phi(f) &=\|g-K_A(f)\|^2_{\cH'}+ \al.|\!|\!|f|\!|\!|, \\
\Phi^{S\cup R} (f;a)&=\Phi(f)+\|f-a\|^2_{(l_2)^n} -\|K_A(f)-K_A(a)\|^2_{\cH'}
\end{align*}
such that $|\!|\!|f|\!|\!|=(|f^1|_{p_1,w_1},\dots,|f^n|_{p_n,w_n})$, 
$|f^i|_{p_i,w_i}=\sum_{\la\in\La_i} w_i|f^i_\la|^{p_i}$
for $f=(f^1,\dots,f^n)\in (l_2)^n$ and $w_i>0$; $i=1,\dots,n$.
Also, for $W=(w_1,\dots,w_n)$, $T=\{t_1,\dots,t_n\}$, $f=(f^1,\dots,f^n)\in 
(l_2)^n$, define operators $\bS_{T,W,\bP}$ by 
$$\bS_{T,W,\bP}(f)=(\{S_{t_1w_1,p_1} 
(f^1_\la)\}_{\la\in\La_1},\dots,\{S_{t_nw_n,p_n} (f_\la^n)\}_{\la\in\La_n})$$
with  functions $S_{tw,p}$ from $R$ to itself given by Lemma \ref{2.1}.

By these assumptions, we will have 

A) $f_m:=\text{minimizer of the functional }\; \Phi^{S\cup 
R}=\bS_{\al,W,\bP}(K^*_A(g-K_A(a))+a)$

B) For all  $h\in(l_2)^n$, one has 
$$\Phi^{S\cup R}(f_m+h;a)\geq \Phi^{S\cup R}(f_m;a)+\|h\|^2_{(l_2)^n}.$$
\end{theorem}

\begin{theorem}\label{3.8}
Pick $f_0$ in $(l_2)^n$. We define the $f_k$ recursively by 
$$f_k:=\arg -\min(\Phi^{S\cup R} (f;f_{k-1})).$$
There is $f^*\in (l_2)^n$ such that $f^*$ is the minimizer of $\Phi$ and 
$\|f_k-f^*\|_{(l_2)^n}\lo 0$. 
\end{theorem}

\begin{theorem}\label{3.9}
Suppose $K_A:(l_2)^n\lo \cH'$ is a linear bounded operator as defined in Lemma 3.5,
 $1\leq p_i\leq 2$ for $i=1,\dots,n$, and either there exists $j$ such 
that $p_j>1$ or $N(K_A)=0$. Also suppose $f^*_{\al,g}$ is the minimizer of the
functional $\Phi_{\al;g}(f)=\|g-K_A(f)\|^2_{\cH'}+\al.|\!|\!|f|\!|\!|$ for 
$\al\in R^n$, $g\in\cH'$. Pick $f_0$ in $(l_2)^n$ and suppose 
$\{g_t\}_{t=1}^{+\infty}$ is a sequence in $\cH'$ such that 
$\|g_t-Kf_0\|\leq\eps_t$, where $\{\eps_t\}_{t=1}^{+\infty}$ is a sequence of 
strictly positive numbers that converges to zero as $t\lo\infty$.
By assuming $\{\al(\eps_t)\}_{t=1}^{+\infty}$ such that 
$\al(\eps_t)=(\al_1(\eps_t),\dots,\al_n(\eps_t))$ and for every $i,j=1,\dots,n$
$$\lim_{t\rig+\infty} \al_i(\eps_t)=0,\quad \lim_{t\rig+\infty} 
\f{\eps_t^2}{\al_i(\eps_t)}=0,\quad \lim_{t\rig+\infty} 
\f{\al_i(\eps_t)}{\al_j(\eps_t)}=1.$$
Then we have
$$\|f^*_{\al(\eps_t),g_t}-f^\dagger\| \lo 0 \quad as \quad t\lo+\infty$$
where $f^\dagger$ is the unique minimal element with regard to
$\boldsymbol{1}\cdot |\!|\!|\cdot|\!|\!|$
in $S=N(K_A)+f_0=\{f\in (l_2)^n; K(f)=K(f_0)\}$. Note that 
$\boldsymbol{1}=(1,\dots,1)$ and $\boldsymbol{1}.|\!|\!|f|\!|\!|=|f^1|_{p_1,w_1}
+\dots+|f^n|_{p_n,w_n}$.
\end{theorem}
\subsection{Inverse imaging with mixed penalties}
In [4], I. Daubechies, G. Teschke have solved the following variational 
problem:
\begin{equation*}
\inf {\cal{F}}_g(v,u),\tag{$16$}
\end{equation*}
where 
$${\cal{F}}_g(v,u)=\| g-K(u+v)\|^2_{L^2(\Om)}+\ga\|v\|^2_{H^{-1}(\Om)} 
+2\al|u|_{B^1_1(L_1(\Om))}$$
such that 
$$\|v\|^2_{H^{-1}(\Om)} \sim \sum_\la 2^{-2|\la|} |v_\la|^2,$$
$$|u|_{B^\beta_p(L_p(\Om))} \sim \left( \sum_\la 2^{|\la|sp} 
|u_\la|^p\right)^{\f{1}{p}} \quad \text{with}\quad s=\beta+1-\f{2}{p}$$
where $u_\la,v_\la$ denote the $\la$-th wavelet coefficients.

Moreover, in [5], C. Demol, M. Defrise have considered the following 
mixed-penalty functional:
\begin{equation*}
\Phi(u,v)=\|A(u+v)-g\|^2 +2\tau |\!|\!| u|\!|\!| +\mu\|v\|^2 \tag{$17$}
\end{equation*}
where $g$ is the image or measurement vector containing $N$ data values and 
$A$ is the $N\ti M$ matrix modeling the imaging process ($A$ is assumed to be 
known). Also, $\|v\|^2=\sum_{m=1}^M |v_m|^2$ denotes the squared $l^2$-norm of 
$v$ and $|\!|\!|u|\!|\!|=\sum_{m=1}^M |u_m|$ denotes the $l^1$-norm of $u$.

Both functionals (16),(17) are special cases of the functional $\Psi$ from 
$\cH\ti \cH$ to $R$ as follows:
\begin{equation*}
\Psi(u,v)=\|k(u+v)-g\|^2+\sum_{\ga\in\Ga} w_{1,\ga} 
|u_\ga|^{p_1}+\sum_{\ga\in\Ga} w_{2,\ga} |v_\ga|^{p_2} \tag{18}
\end{equation*}
To solve this functional, we need to define the linear bounded operator $L$ on 
${\boldsymbol{\cH}}=\cH\ti\cH$ as follows:
\begin{align*}
L:\boldsymbol{\cH}&\lo \cH'\\
L(u,v) &=k(u+v)
\end{align*}
By this definition, we will have
\begin{align*}
&L^*:{\cH}'\lo{\boldsymbol{\cH}}\\
&L^*(y)=(K^*(y),K^*(y))
\end{align*}
because 
\begin{align*}
<(u,v),L^*(y)>_{\boldsymbol{\cH}} 
&=<(u,v),(K^*(y),K^*(y))>_{\boldsymbol{\cH}}\\
&=<u,K^*(y)>_{\cH} +<v,K^*(y)>_{\cH}\\
&=<K(u),y>_{\cH'} +<K(v),y>_{\cH'}\\
&=<K(u+v),y>_{\cH'}=<L(u,v),y>_{\cH'}
\end{align*}
With the assumption of the Hilbert space ${\boldsymbol{\cH}}=\cH\ti\cH$
 with the basis 
$\{(\var_\ga,0)\}_{\ga\in\Ga} \cup \{(0,\var_\ga)\}_{\ga\in\Ga}$ and dividing 
the basis into two parts $\{(\var_\ga,0)\}_{\ga\in\Ga}$ and 
$\{(0,\var_\ga)\}_{\ga\in\Ga}$, and the linear bounded operator $L$ from 
${\boldsymbol{\cH}}$ to $\cH'$, and also by considering $W=W^1\cup 
W^2=\{w_{1,\ga}\}_{\ga\in\Ga}\cup \{w_{2,\ga}\}_{\ga\in\Ga}$, we will have the 
functional $\Phi$ in section 2 as follows.
\begin{align*}
\Phi(f) &=\|L(f)-g\|^2+\sum_{\ga\in\Ga} w_{1,\ga} 
|<f,(\var_{\ga,0})>|^{p_1}+\sum_{\ga\in\Ga} w_{2,\ga} |<f,(0,\var_\ga)>|^{p_2} 
\\
\Phi(u,v)&=\|L(u,v)-g\|^2+\sum_{\ga\in\Ga} 
w_{1,\ga}|<(u,v),(\var_{\ga,0})>|^{p_1}+ \sum_{\ga\in\Ga} w_{2,\ga} 
|<(u,v),(0,\var_\ga)>|^{p_2} \\
\Phi(u,v) &=\|K(u+v)-g\|^2+\sum_{\ga\in\Ga} 
w_{1,\ga}|<u,\var_\ga>|^{p_1}+\sum_{\ga\in\Ga}w_{2,\ga}|<v,
\var_\ga>|^{p_2}\\
\Phi(u,v)&=\|K(u+v)-g\|^2+ |\!|\!|u|\!|\!|^{p_1}_{W^1,p_1} 
+|\!|\!|v|\!|\!|^{p_2}_{W^2,p_2}
\end{align*}
Thus
$$\Phi(u,v) =\Psi(u,v) \quad \text{for} \quad (u,v)\in{\boldsymbol{\cH}}.$$
By the theorems in section 2, we conclude that the sequence 
$\{f^n\}_{n=1}^\infty$ is strongly covergent to the minimizer of the 
functional $\Psi$ in (18), as follows:
$$f^n=\bS_{W,\bP}(f^{n-1}+L^*(g-L(f^{n-1})))$$
such that, for $\bP=\{p_1,p_2\}$, $W=\{w_{1,\ga}\}_{\ga\in\Ga} \cup 
\{w_{2,\ga}\}_{\ga\in\Ga}$. By elimination and substitution, we will have the 
sequence $\{f^n\}_{n=1}^{+\infty}$ as follows.
\begin{align*}
f^n=(u^n,v^n) &=\bS_{W,\bP} ((u^{n-1}, v^{n-1})+L^*(g-L(u^{n-1},v^{n-1}))) \\
&=\bS_{W,\bP} ((u^{n-1},v^{n-1})+(K^*(g-K(u^{n-1}+v^{n-1})), 
K^*(g-K(u^{n-1}+v^{n-1})))) \\
&=\bS_{W,\bP} (u^{n-1}+K^*(g-K(u^{n-1}+v^{n-1})), 
v^{n-1}+K^*(g-K(u^{n-1}+v^{n-1}))) \\
&=\sum_{\ga\in\Ga} S_{w_{1,\ga},p_1} ((u^{n-1}+K^*(g-K(u^{n-1}+v^{n-1})), 
v^{n-1}+K^*(g-K(u^{n-1}+v^{n-1}))), (\var_\ga,0))(\var_\ga,0) \\
&+\sum_{\ga\in\Ga} S_{w_{2,\ga},p_2} ((u^{n-1}+K^*(g-K(u^{n-1}+v^{n-1})), 
v^{n-1}+K^*(g-K(u^{n-1}+v^{n-1}))), (0,\var_\ga))(0,\var_\ga)\\
&=(\sum_{\ga\in\Ga} S_{w_{1,\ga},p_1} (u^{n-1}+K^*(g-K(u^{n-1}+v^{n-1})), 
\var_\ga) \var_\ga\\
&, \sum_{\ga\in\Ga} S_{w_{2,\ga},p_2} 
(v^{n-1}+K^*(g-K(u^{n-1}+v^{n-1})),\var_\ga)\var_\ga) 
\end{align*}
Therefore, we have $u^n,v^n$ as follows:
\begin{align*}
u^n &=\sum_{\ga\in\Ga} S_{w_{1,\ga},p_1} 
(u^{n-1}+K^*(g-K(u^{n-1}+v^{n-1})),\var_\ga) \var_\ga,\\
v^n&=\sum_{\ga\in\Ga} S_{w_{2,\ga},p_2} 
(v^{n-1}+K^*(g-K(u^{n-1}+v^{n-1})),\var_\ga)\var_\ga.
\end{align*}
\subsection{Linear inverse problems with multi-constraints}
In this subsection, we consider a generalization for [3] that differes from 
the one given in section 2.

Minimization of the following functionals has been considered in [1,2,3,6].
\begin{align*}
& \|Kf-g\|^2_Y+\mu\|f\|_X^2\\
&\|I(f)-g\|^2_{L_2(I)} +\la\|f\|^2_{W^\beta(L_2(I))}; 
\|f\|^2_{W^\beta(L_2(I))} = 
\sum_{0\leq k} \sum_{j\in \BZ^2_K} \sum_{\psi\in\psi_k} 2^{2\beta k} 
|c_{j,k,\psi}|^2; \beta <\f{1}{2} \\
&\|I(f)-g\|^2_{L_2(I)}+\la\|f\|^\tau_{B^\beta_\tau(L_\tau(I))}; 
\|f\|^\tau_{B^\beta_\tau((L_\tau(I))}=
\sum_{0\leq k}\sum_{j\in\BZ_k^2} \sum_{\psi\in\psi_k} |c_{j,k,\psi}|^\tau\\
&\|Kf-g\|^2+\sum_{\la\in\La} 2^{\sig p|\la|} |<f,\Psi_\la>|^p
\end{align*}
All of above functionals are special cases of the following functional 
$$\|Kf-g\|^2 +\la\|f\|_Y^s$$
where $Y$ is a space that measures the {\it smoothness} of the approximations 
$f$, $\la$ is a positive parameter, and $s$ is an exponent that is chosen to 
make the computations (and analysis) easier. If the positive parameter $\la$ 
is large, then the smoothness of $f$ is important; if it is small, the 
approximation error between $g$ and $K(f)$ is important. 

The point arising at this time is that if we can simultaneously consider two 
or more constraints in a minimization problem, we can potentially obtain 
better results.

With this account, we tried to show the importance of the minimization problem 
of the following functional:
$$\|Kf-g\|^2_{\cH'} +\la_1\|f\|_1^{s_1}+\dots+\la_n \|f\|_n^{s_n}$$
or 
$$\|Kf-g\|^2_{\cH'} +|\!|\!|f|\!|\!|^{p_1}_{W^1,p_1}+\dots+|\!|\!|f|\!|\!|^{p_n}
_{W^n,p_n}$$
where 
$$|\!|\!|f|\!|\!|_{W^i,p_i} =(\sum_{\ga\in\Ga} 
w_{i,\ga}|<f,\var_\ga>|^{p_i})^{\f{1}{p_i}}$$
for $1\leq i\leq n$; $1\leq p_i\leq 2$, $W^i=\{w_{i,r}\}_{\ga\in\Ga}$, 
$\{\var_\ga\}_{\ga\in\Ga}$ is an orthonormal basis for Hilbert space $\cH$.

We call the new constraint $\sum_{i=1}^n |\!|\!|f|\!|\!|^{p_i}_{W^i,p_i}$ {\it 
multi-constraints}, which are different from the constraint $\sum_{i=1}^n 
|\!|\!|f|\!|\!|^{p_i}_{W_i,p_i}$ introduced in section 2, which we called {\it 
mixed 
multi-constraints}. We cannot obtain the solution of this problem directly (as 
was done in subsection 3.2) from the theorems in section 2, but we can solve 
this problem by modifying some lemmas and theorems of section 2 and through a 
procedure similar to what was carried out there.

Except for Lemma 2.1, Proposition 2.2, part A of Lemma 2.8 and part C of
Theorem 2.15 which are rewritten as follows, all the modifications are trivial.
\begin{lemma}
The minimizer of the function $M(x)=x^2-2bx+\sum_{i=1}^n c_i |x|^{p_i}$ for 
$\forall i, 1\leq i\leq n$; $p_i\geq 1$, $c_i>0$ is 
$S_{(c_1,\dots,c_n),(p_1,\dots,p_n)}(b)$, where the function 
$S_{(c_1,\dots,c_n),(p_1,\dots,p_n)}$ from $R$ to itself is defined by 
$$S_{(c_1,\dots,c_n),(p_1,\dots,p_n)} (t) =
\begin{cases}
F^{-1}(t) & B=\phi\\
F_1^{-1}(t) & B\neq\phi, t\in \left(\f{\sum_{i\in B}c_i}{2},+\infty\right) \\
0 & B\neq\phi, t\in \left[-\f{\sum_{i\in B}c_i}{2}, \f{\sum_{i\in B} 
c_i}{2}\right] \\
F_2^{-1}(t) & B\neq\phi, t\in\left(-\infty, \f{-\sum_{i\in B}c_i}{2}\right) 
\end{cases}$$
where $B=\{i| p_i=1\}$ and the functions $F_1,F_2,F$ are defined by 
\begin{align*}
&F(x)=x+\Sign x\f{\sum_{i=1}^n p_ic_i |x|^{p_i-1}}{2} \quad \text{for}\quad 
x\in R,\\
&F_1(x)=x+\f{\sum_{i\in B}c_i+\sum_{i\not\in B}p_ic_i|x|^{p_i-1}}{2}
 \quad \text{for}\quad x>0,\\
&F_2(x)=x-\f{\sum_{i\in B}c_i+\sum_{i\not\in B}p_ic_i|x|^{p_i-1}}{2}
 \quad \text{for}\quad x<0.
\end{align*}
\end{lemma}

\begin{proof}
Since, $\lim_{x\rig\pm\infty} M(x)=+\infty$ and $M$ is a continuous function, 
$M$ has minimizer on $R$.

For $B=\phi$, $M$ is differentiable, and the minimization reduces to solving 
the variational equation
$$b=x+\Sign x\f{\sum_{i=1}^n p_i c_i|x|^{p_i-1}}{2}$$
since, the real function
$$F(x)=x+\Sign x\f{\sum_{i=1}^n p_ic_i|x|^{p_i-1}}{2}$$
is a one-to-one map from $R$ to itself, the minimizer is $F^{-1}(b)$, i.e.
$$S(b)=F^{-1}(b).$$
For$B\neq\phi$, $M$ is differentiable only if $x\neq 0$; then either minimizer 
is $x=0$ or minimizer is an $x\neq 0$ such that $M'(x)=0$. Note that $F_1,F_2$ 
are one-to-one maps such that 
\begin{align*}
Domain F_1=(0,+\infty),\quad & \quad Rang F_1=(\f{\sum_{i\in 
B}c_i}{2},+\infty) \\
Domain F_2=(-\infty,0),\quad & \quad Rang F_2=(-\infty,\f{-\sum_{i\in 
B}c_i}{2}).
\end{align*}
\paragraph{Case 1:} $b\in\left( \f{\sum_{i\in B} c_i}{2},+\infty\right)$.

For $x>0$, we have
$$M'(x)=2x-2b+\sum_{i\in B}c_i+\sum_{i\not\in B} p_ic_i |x|^{p_i-1}$$
If $M'(x)=0$, then 
$$b=x+\f{\sum_{i\in B} c_i+\sum_{i\not\in B} p_ic_i|x|^{p_i-1}}{2} =F_1(x)$$
Since, $b\in Rang F_1$, we have
$$\exists ! x_0\in (0,+\infty); x_0=F_1^{-1}(b)$$
Therefore, $$M'(F_1^{-1}(b))=0.$$
For $x<0$, we have
$$M'(x)=2x-2b-\sum_{i\in B}c_i-\sum_{i\not\in B} p_ic_i |x|^{p_i-1}$$
If $M'(x)=0$, then
$$b=x-\f{\sum_{i\in B} c_i+\sum_{i\not\in B}p_ic_i |x|^{p_i-1}}{2}=F_2(x)$$
since, $b\not\in Rang F_2$,
$$\not\exists x\in (-\infty,0); b=F_2(x)$$
Therefore, 
$$\not\exists x\in (-\infty,0); M'(x)=0.$$
By above argumnet, we conclude $S(b)\in\{0,F_1^{-1}(b)\}$. By the following 
fact, we will have
$$M(F_1^{-1}(b))\leq 0 = M(0)$$
Then
$$S(b)=F_1^{-1}(b).$$
\paragraph{Fact:} If $M$ is a continuous function such that 

(i) $\lim_{x\rig \pm\infty} M(x)=+\infty$.

(ii) $M(0)=0$.

(iii) $(0,+\infty)\subs Domain M'$ and for every $x,x'\in (0,+\infty)$ such 
that $x'<x$ we have $M'(x')<M'(x)$.

(iv) $\exists ! x_m\in (0,+\infty); M'(x_m)=0$.\\
Then 
$$M(x_m)\leq 0.$$
If $M(x_m)>0$, then $\exists \del>0; \forall x\in (x_m-\del,x_m+\del)$ 
$M(x)>0$, therefore 
\begin{align*}
& \exists x_0\in (0,x_m); M(x_0)>0\\  
\Lo & M'(x_0) <M'(x_m)=0\\
\Lo & \forall x\in (0,x_0) M'(x)<M'(x_0)<0\\
\Lo & \forall x\in (0,x_0) M(x)>M(x_0)>0
\end{align*}
This is a contradiction, because $M(0)=0$ and $f$ is a continuous function.
\paragraph{Case 2:} $b\in (-\infty, \f{-\sum_{i\in B} c_i}{2})$.

By an argument as case 1, we conclude $S(b)=F_2^{-1}(b)$.
\paragraph{Case 3:} $b\in \left[ \f{-\sum_{i\in B}c_i}{2}, \f{\sum_{i\in 
B}c_i}{2}\right]$.

We have $b\not\in Rang F_1$, $b\not\in Rang F_2$, then
$$\not\exists x\neq 0;\quad M'(x)=0$$
Therefore,
$$S(b)=0.$$
\end{proof}
\begin{proposition}
Suppose $K:\cH\lo \cH'$ is an operator, with $\|KK^*\|<1$, 
$(\var_\ga)_{\ga\in\Ga}$ is an orthonormal basis for $\cH$, and 
$W^i=(w_{i,\ga})_{\ga\in\Ga}$ are a sequences such that $\forall \ga\in\Ga$, 
$w_{i,\ga}>c>0$ for $1\leq i\leq n$. Further suppose $g$ is an element of 
$\cH'$. Let $W=W^1\cup W^2\cup\dots\cup W^n$, $\bP=\{p_1,\dots,p_n\}$ such 
that $p_i\geq 1$ for $1\leq i\leq n$. Choose $a\in\cH$ and define the 
functional $\Phi_{W,\bP}^{S\cup R}(f;a)$ on $\cH$ by 
$$\Phi_{W,\bP}^{S\cup R} (f;a) =\|Kf-g\|^2+\sum_{i=1}^n \sum_{\ga\in\Ga} 
w_{i,\ga}|f_\ga|^{p_i}+\|f-a\|^2- \|K(f-a)\|^2.$$
Also, define operators $\bS_{W,\bP}$ by 
\begin{equation*}
\bS_{W,\bP}(h)=\sum_{\ga\in\Ga} 
S_{(w_{1,\ga},\dots,w_{n,\ga}),(p_1,\dots,p_n)}(h_\ga) \var_\ga \tag{$19$}
\end{equation*}
with functions $S_{(w_{1,\ga},\dots,w_{n,\ga}),(p_1,\dots,p_n)}$ from $R$ to 
itself given by Lemma 3.10.

By these assumptions, we will have

A) $f_{\min}=minimizer$ of the functional $\Phi_{W,\bP}^{S\cup R} 
=\bS_{W,\bP}(a+K^*(g-Ka))$

B) for all $h\in\cH$, one has
$$\Phi_{W,\bP}^{S\cup R} (f_{\min}+h;a) \geq \Phi_{W,\bP}^{S \cup R}
(f_{\min};a)+\|h\|^2.$$
\end{proposition}

\begin{proof}
A)
$$\Phi_{W,\bP}^{S\cup R}(f;a)=\sum_{\ga\in\Ga} [f_\ga^2-2f_\ga 
(a+K^*g-K^*Ka)_\ga+\sum_{i=1}^n w_{i,\ga}|f_\ga|^{p_i} 
]+\|g\|^2+\|a\|^2-\|Ka\|^2$$
By Lemma 3.10, we have
\begin{align*}
f_{\min} &=\sum_{\ga\in\Ga} S_{(w_{1,\ga},\dots,w_{n,\ga}),(p_1,\dots,p_n)} 
((a+K^*g-K^*Ka)_\ga)\var_\ga\\
&=\bS_{W,\bP}(a+K^*g-K^*Ka)
\end{align*}

B)
\begin{multline*}
\Phi_{W,\bP}^{S\cup R} (f+h;a)-\Phi_{W,\bP}^{S\cup R} (f;a) \\
=\sum_{\ga\in\Ga} ((\sum_{i=1}^n w_{i,\ga}|f_\ga+h_\ga|^{p_i} 
-w_{i,\ga}|f_\ga|^{p_i})+2h_r (f-a-K^*g-K^*Ka)_\ga)+\|h\|^2
\end{multline*}
The cases $B=\phi$ and $B\neq\phi$ should be treated slightly differently.
\paragraph{Case 1:} For $B=\phi$, we have 
$$S(t)=t+\Sign t\f{\sum_{i=1}^n w_{i,\ga} p_i|t|^{p_i-1}}{2}$$
Then, 
$$(a+K^*g-K^*Ka)_\ga=f_\ga+\Sign f_\ga \f{\sum_{i=1}^n 
w_{i,\ga}p_i|f_\ga|^{p_i-1}}{2}$$
Therefore, 
$$2(f-a-K^*(g-K(a)))_\ga= -\Sign f_\ga (\sum_{i=1}^n 
w_{i,\ga}p_i|f_\ga|^{p_i-1})$$
for $f_\ga\neq 0$, there is $\al\in R$ such that $h_\ga=\al f_\ga$, then 
\begin{align*}
& \sum_{i=1}^n (w_{i,\ga}|f_\ga+h_\ga|^{p_i}- w_{i,\ga}|f_\ga|^{p_i})+2 h_\ga 
(f-a-K^*(g-Ka))_\ga \\
=&\sum_{i=1}^n (w_{i,\ga}|f_\ga+h_\ga|^{p_i} -w_{i,\ga}|f_\ga|^{p_i})- h_\ga 
\Sign f_\ga (\sum_{i=1}^n w_{i,\ga}p_i|f_\ga|^{p_i-1})\\
=& \sum_{i=1}^n (w_{i,\ga}|f_\ga+\al f_\ga|^{p_i} - w_{i,\ga}|f_\ga|^{p_i} - 
\al w_{i,\ga} p_i f_\ga \Sign f_\ga |f_\ga|^{p_i-1}) \\
=&\sum_{i=1}^n [w_{i,\ga}|f_\ga|^{p_i} (|1+\al|^{p_i}-1-p_i\al)]\geq 0
\end{align*}

If $f_\ga=0$, then
\begin{align*}
&\sum_{i=1}^n w_{i,\ga} |f_\ga+h_\ga|^{p_i} -w_{i,\ga}|f_\ga|^{p_i} +2h_\ga 
(f-a-K^*(g-Ka))_\ga\\
=&\sum_{i=1}^n w_{i,\ga} |f_\ga+h_\ga|^{p_i} -w_{i,\ga} |f_\ga|^{p_i} -p_i 
w_{i,\ga} h_\ga \Sign f_\ga|f_\ga|^{p_i-1} \\
=&\sum_{i=1}^n w_{i,\ga}|h_\ga|^{p_i} \geq 0
\end{align*}
\paragraph{Case 2:}  $B\neq \phi$.

Define now $\Ga_0=\{\ga\in\Ga: f_\ga=0\}$, and $\Ga_1=\Ga\backslash \Ga_0$. We 
will have
\begin{align*}
& \sum_{\ga\in\Ga_0} \left[\left( \sum_{i\in B} w_{i,\ga}|f_\ga+h_\ga| 
-w_{i,\ga}|f_\ga|\right) +\left( \sum_{i\not\in B} w_{i,\ga} 
|f_\ga+h_\ga|^{p_i} -w_{i,\ga}|f_\ga|^{p_i}\right) + 2h_\ga 
(f-a-K^*(g-Ka))_\ga\right]\\
=&\sum_{\ga\in\Ga_0} \left[ \left( \sum_{i\in B} w_{i,\ga}|h_\ga|-2h_\ga 
(a+K^*(g-Ka))_\ga\right) +\left( \sum_{i\not\in B} w_{i,\ga} 
|h_\ga|^{p_i}\right)\right]
\end{align*}
For $\ga\in\Ga_0$, $|a_\ga+[K^*(g-Ka)]_\ga| < \f{\sum_{i\in B}w_{i,\ga}}{2}$, 
so that 
$\sum_{i\in B} w_{i,\ga}|h_\ga| -2h_\ga (a_\ga+[K^*(g-Ka)]_\ga)\geq 0$.

If $\ga\in\Ga_1$, we distinguish two cases, acording to the sign of $f_\ga$. 
We discuss here only the case $f_\ga>0$; the similar case $f_\ga<0$ is left to 
the reader. For $f_\ga>0$, we have 
$$(a+K^*(g-Ka))_\ga= f_\ga+\f{\sum_{i\in B} w_{i,\ga} +\sum_{i\not\in B} 
w_{i,\ga} p_i|f_\ga|^{p_i-1}}{2}$$
Consequently 
\begin{align*}
& \sum_{\ga\in\Ga} \left[\left( \sum_{i\in B} w_{i,\ga} |f_\ga+h_\ga| 
-w_{i,\ga}|f_\ga|\right) +\left( \sum_{i\not\in B} w_{i,\ga} 
|f_\ga+h_\ga|^{p_i}-w_{i,\ga}|f_\ga|^{p_i}
 \right) +2h_\ga(f-a-K^*(g-Ka))_\ga\right]\\
=&\sum_{\ga\in\Ga} \sum_{i\in B}\left(
w_{i,\ga} |f_\ga +h_\ga| -w_{i,\ga} |f_\ga| 
-h_\ga w_{i,\ga}\right) \\
+& \sum_{\ga\in\Ga}\sum_{i\not\in B} \left( w_{i,\ga} |f_\ga+h_\ga|^{p_i} 
-w_{i,\ga} |f_\ga|^{p_i} -w_{i,\ga} p_i  h_\ga|f_\ga|^{p_i-1}\right)
\end{align*}
We have,
\begin{align*}
& \sum_{i\in B} w_{i,\ga} |f_\ga+h_\ga| -w_{i,\ga}|f_\ga|-h_\ga w_{i,\ga} \\
&=\sum_{i\in B} w_{i,\ga} \left( |f_\ga+h_\ga| -(f_\ga+h_\ga)\right) \geq 0
\end{align*}
Also, there is $\al\in R$ such that $h_\ga=\al f_\ga$, then 
\begin{align*}
& \sum_{i\not\in B} w_{i,\ga} |f_\ga+h_\ga|^{p_i} - w_{i,\ga} |f_\ga|^{p_i} 
-w_{i,\ga} p_i h_\ga |f_\ga|^{p_i-1}\\
=&\sum_{i\not\in B} w_{i,\ga} |f_\ga+\al f_\ga|^{p_i} -w_{i,\ga} |f_\ga|^{p_i} 
-\al w_{i,\ga} p_i f_\ga |f_\ga|^{p_i-1}\\
=& \sum_{i\not\in B} w_{i,\ga} |f_\ga|^{p_i} (|1+\al|^{p_i} -1-\al p_i)\geq 0.
\end{align*}
\end{proof}

\begin{lemma}
The operators $\bS_{W,\bP}$ are nonexpansive, i.e., 
$$\forall v,v'\in \cH, \| \bS_{W,\bP} (v)-\bS_{W,\bP}(v')\| \leq \|v-v'\|$$
\end{lemma}

\begin{proof}
As shown  by (19),
$$\|\bS_{W,\bP}(v)-\bS_{W,\bP}(v')\|^2= \sum_{\ga\in\Ga} 
|S_{(w_{1,\ga},\dots,w_{n,\ga}),(p_1,\dots,p_n)} (v_\ga)- 
S_{(w_{1,\ga},\dots,w_{n,\ga}),(p_1,\dots,p_n)}(v'_\ga)|^2$$
which means that it suffices to show that,  $\forall x,x'\in R$, and all 
$(w_1,\dots,w_n),(p_1,\dots,p_n)$ such that $w_i\geq 0$, $p_i\geq 1$ for 
$i=1,\dots,n$,
\begin{equation*}
\left| S_{(w_1,\dots,w_n),(p_1,\dots,p_n)} (x)- 
S_{(w_1,\dots,w_n),(p_1,\dots,p_n)}(x')\right| \leq |x-x'| \tag{$20$}
\end{equation*}
If $B=\phi$, then $S$ is the inverse of the function $F$. $F$ is 
differentiable and we have 
$$F'(x)=1+\f{\sum_{i=1}^n w_i p_i(p_i-1)|x|^{p_i-2}}{2}$$
thus 
$$|F'(x)|>1$$
and (20) follows immediately in this case.
For $B\neq\phi$, we use another argument. For the sake of definiteness, let us 
assume $x>x'$. We will just check all the possible cases:
\paragraph{Case 1:} $x,x'>\f{\sum_{i\in B} w_i}{2}$. In this case, we have
\begin{align*}
& S(x),S(x')>0, \quad S(x')<S(x) , \\
& x=S(x)+\f{\sum_{i\in B} w_i+ \sum_{i\not\in B} w_i p_i| S(x)|^{p_i-1}}{2},\\
& x'=S(x')+\f{\sum_{i\in B} w_i+\sum_{i\not\in B} w_ip_i |S(x')|^{p_i-1}}{2}.
\end{align*}
Then,
$$0<S(x)-S(x')=x-x'+\f{\sum_{i\not\in B} w_i 
p_i(|S(x')|^{p_i-1}-|S(x)|^{p_i-1})}{2}\leq x-x'.$$
\paragraph{Case 2:} $\f{-\sum_{i\in B} w_i}{2} \leq x,x'\leq \f{\sum_{i\in B} 
w_i}{2}$. In this case, we have
$$S(x)=S(x')=0.$$
Then, 
$$|S(x)-S(x')|=|0-0| =0 <|x-x'|.$$
\paragraph{Case 3:} $x,x'<\f{-\sum_{i\in B} w_i}{2}$. In this case, we have
\begin{align*}
& S(x),S(x')<0, \quad S(x')<S(x) , \\
& x=S(x)-\f{\sum_{i\in B} w_i+ \sum_{i\not\in B} w_i p_i| S(x)|^{p_i-1}}{2},\\
& x'=S(x')-\f{\sum_{i\in B} w_i+\sum_{i\not\in B} w_ip_i |S(x')|^{p_i-1}}{2}.
\end{align*}
Then,
$$0<S(x)-S(x')=x-x'+\f{\sum_{i\not\in B} w_i 
p_i(|S(x)|^{p_i-1}-|S(x')|^{p_i-1})}{2}< x-x'.$$
\paragraph{Case 4:} $x>\f{\sum_{i\in
 B} w_i}{2}, -\f{\sum_{i\in B}w_i}{2} \leq x'\leq \f{\sum_{i\in B} 
w_i}{2}$. In this case, we have
\begin{align*}
& S(x)>0, \\
& x=S(x)+\f{\sum_{i\in B} w_i+\sum_{i\not\in B} p_iw_i|S(x)|^{p_i-1}}{2},\\
& S(x')=0
\end{align*}
Then,
$$0<S(x)-S(x')=S(x)=x-\f{\sum_{i\in B} w_i+\sum_{i\not\in B} p_iw_i 
|S(x)|^{p_i-1}}{2} <x-\f{\sum_{i\in B}w_i}{2} \leq x-x'.$$
\paragraph{Case 5:} $-\f{\sum_{i\in B}w_i}{2}\leq x\leq \f{\sum_{i\in 
B}w_i}{2}, x'<-\f{\sum_{i\in B} w_i}{2}$. In this case, we have
\begin{align*}
& S(x')<0, \\
& x'=S(x')-\f{\sum_{i\in B} w_i+\sum_{i\not\in B} p_iw_i|S(x')|^{p_i-1}}{2},\\
& S(x)=0
\end{align*}
Then,
$$0<S(x)-S(x')=-S(x')=-x'-\f{\sum_{i\in B} w_i+\sum_{i\not\in B} p_iw_i 
|S(x')|^{p_i-1}}{2} <-\f{\sum_{i\in B}w_i}{2}-x' \leq x-x'.$$
\paragraph{Case 6:} $x>\f{\sum_{i\in B} w_i}{2}$, $x'<\f{-\sum_{i\in 
B}w_i}{2}$. In this case, we have 
\begin{align*}
& S(x)>0, S(x')<0 \\
& x=S(x)+\f{\sum_{i\in B} w_i+\sum_{i\not\in B} p_iw_i|S(x)|^{p_i-1}}{2},\\
& x'=S(x')-\f{\sum_{i\in B} w_i+\sum_{i\not\in B}w_i p_i|S(x')|^{p_i-1}}{2}.
\end{align*}
Then,
$$0<S(x)-S(x')=x-x'-\sum_{i\in B} w_i-\f{\sum_{i\not\in B}
w_ip_i (|S(x')|^{p_i-1}+|S(x')|^{p_i-1})}{2}<x-x'.$$
\end{proof}

\begin{lemma}
If for some $a\in\cH$ and some sequence $(v^n)_{n\in\BN}$, $w-\lim_{n\rig 
+\infty} v^n=0$ and $\lim_{n\rig+\infty} \|\bS_{W,\bP} 
(a+v^n)-\bS_{W,\bP}(a)-v^n\|=0$
 such that $p_i\in [1,2]$ for $i=1,\dots,n$ then 
$\|v^n\|\lo 0$ for $n\lo\infty$. 
\end{lemma}

\begin{proof}
The argument of the proof is slightly different for the cases $B=\phi$ and 
$B\neq\phi$. If $B=\phi$, then $S(x)=F^{-1}(x)$ and since $p_i\in [1,2]$, we 
have
$$F'(x)=1+\f{\sum_{i=1}^n w_{i,\ga} p_i(p_i-1)|x|^{p_i-2}}{2} 
<1+\f{\sum_{i=1}^n w_{i,\ga} p_i(p_i-1)}{2(2b)^{2-p_i}}$$
for $x\in [-b,b]$ where $b>0$. Therefore, in this case, we will have an 
argument similar to what has been done for case $p>1$ in [3].

For $B\neq\phi$, we define a finite set $\Ga_0\subs\Ga$ so that 
$\sum_{\ga\in\Ga\backslash \Ga_0} |a_\ga|^2\leq (\f{c}{4})^2$, where $c$ is 
the uniform lower bound on the $\sum_{i=1}^n w_{i,\ga}$. Because this is a 
finite set, the weak convergence of the $v^n$ implies that $\sum_{\ga\in\Ga_0} 
|v_\ga^n|^2\underset{n\lo\infty}{\lo}
0$,  so that we can concentrate on 
$\sum_{\ga\in\Ga\backslash \Ga_0} |v_\ga^n|^2$ only.

For each $n$, we split $\Ga_1=\Ga\backslash \Ga_0$ into two subsets: 
$\Ga_{1,n}=\{\ga\in\Ga_1: |v_\ga^n+a_\ga| \leq \f{\sum_{i\in B} 
w_{i,\ga}}{2}\}$ and $\tilde{\Ga}_{1,n}=\Ga_1 \backslash \Ga_{1,n}$. If 
$\ga\in\Ga_{1,n}$, then $S(a_\ga+v_\ga^n)= S(a_\ga)=0$ (since 
$|a_\ga|<\f{c}{4}\leq \f{\sum_{i\in B} w_{i,\ga}}{2})$, so that 
$|v^n_\ga-S(a_\ga+v^n_\ga)+S(a_\ga)|=|v_\ga^n|$. 
It follows that 
$$\sum_{\ga\in\Ga_{1,n}} |v_\ga^n|^2\leq \sum_{\ga\in\Ga} |v_\ga^n 
-S(a_\ga+v_\ga^n)+S(a_\ga)|^2\lo 0 \quad \text{as}\quad n\lo\infty.$$
It remains to prove only that the remaining sum, 
$\sum_{\ga\in\tilde{\Ga}_{1,n}} |v_\ga^n|^2$, also tends to 0 as $n\lo\infty$. 
If $\ga\in\Ga_1$ and $v_\ga^n+a_\ga> \f{\sum_{i\in B}w_{i,\ga}}{2}$, then 
\begin{align*}
|v_\ga^n-S(a_\ga+v_\ga^n)+S(a_\ga)| &=|v_\ga^n-S(a_\ga+v_\ga^n)| \\
&\geq v_\ga^n -S(a_\ga+v_\ga^n) \\
&=-a_\ga+\f{\sum_{i\in B} w_{i,\ga} +\sum_{i\not\in B} w_{i,\ga} 
p_i|S(a_\ga+v_\ga^n)|^{p_i-1}}{2}\\
&>-a_\ga+\f{\sum_{i\in B} w_{i,\ga}}{2} \\
&>-\f{c}{4}+\f{\sum_{i\in B} w_{i,\ga}}{2}>-\f{c}{4} +\f{c}{2}=\f{c}{4}.
\end{align*}
If $\ga\in\Ga_1$ and $v_\ga^n+a_\ga<\f{\sum_{i\in B} w_{i,\ga}}{2}$, then 
\begin{align*}
|v_\ga^n-S(a_\ga+v_\ga^n)+S(a_\ga)| &=|v_\ga^n-S(a_\ga+v_\ga^n)| \\
&\geq S(a_\ga+v_\ga^n) -v^n_\ga\\
&=a_\ga+\f{\sum_{i\in B} w_{i,\ga} +\sum_{i\not\in B} w_{i,\ga} 
p_i|S(a_\ga+v_\ga^n)|^{p_i-1}}{2} \\
&>a_\ga+\f{\sum_{i\in B} w_{i,\ga}}{2} \\
&>-\f{c}{4}+\f{\sum_{i\in B} w_{i,\ga}}{2}>-\f{c}{4}+\f{c}{2}=\f{c}{4}.
\end{align*}
This implies that 
$$\sum_{\ga\in\tilde{\Ga}_{1,n}} |v_\ga^n-S(a_\ga+v_\ga^n)+S(a_\ga)|^2 \geq 
(\f{c}{4})^2 Card (\tilde{\Ga}_{1,n});$$
since $\|v^n-\bS_{W,\bP}(a+v^n)+\bS_{W,\bP}(a)
 \| \underset{n\rig+\infty}{\lo} 0$, we know on 
the other hand that 
$$\sum_{\ga\in\tilde{\Ga}_{1,n}} |v_\ga^n-S(a_\ga+v_\ga^n)+S(a_\ga)|^2< 
(\f{c}{4})^2$$
when $n$ exceeds some threshold $N$, which implies that $\tilde{\Ga}_{1,n}$ is 
empty when $n>N$. Consequently, $\sum_{\ga\in\tilde{\Ga}_{1,n}}|v_\ga^n|^2=0$ 
for $n>N$. This completes the proof for the case $B\neq\phi$.
\end{proof}

\end{document}